\documentclass[10pt,reqno]{amsart}

\usepackage{amsmath,amscd,amssymb,amsthm}
\usepackage{tikz-cd}
\usepackage{comment}
\usepackage[all,cmtip]{xy}
\usepackage{pgf}
\usepackage{tikz}

\makeatletter
\def\th@plain{%
  \thm@notefont{}
  \itshape 
}
\def\th@definition{%
  \thm@notefont{}
  \normalfont 
}
\makeatother

\usepackage[pdftex,hyperindex]{hyperref}

\usepackage{enumerate}
\usepackage[textsize=tiny]{todonotes}

\theoremstyle{definition}

\newtheorem{proposition}{Proposition}[section]
\newtheorem{lemma}[proposition]{Lemma}
\newtheorem{theorem}[proposition]{Theorem}
\newtheorem{corollary}[proposition]{Corollary}
\newtheorem{conjecture}[proposition]{Conjecture}

\newtheorem{remark}[proposition]{Remark}
\newtheorem{definition}[proposition]{Definition}

\numberwithin{equation}{section} 

\newcommand{\N}{\mathbb{N}}
\newcommand{\R}{\mathbb{R}}
\newcommand{\C}{\mathbb{C}}

\newcommand{\Z}{\mathbb{Z}}
\newcommand{\Q}{\mathbb{Q}}

\newcommand{\pr}{\mathbb{P}}

\newcommand{\sslash}{\mathbin{/\mkern-6mu/}}

\newcommand{\B}{\mathcal{B}}
\newcommand{\scL}{\mathcal{L}}

\newcommand{\scI}{\mathcal{I}}
\newcommand{\D}{\mathcal{D}}
\newcommand{\scH}{\mathcal{H}}
\newcommand{\scM}{\mathcal{M}}

\newcommand{\X}{\mathcal{X}}
\newcommand{\Y}{\mathcal{Y}}

\newcommand{\scO}{\mathcal{O}}

\renewcommand{\L}{\mathcal{L}}
\newcommand{\J}{\mathcal{J}}
\newcommand{\M}{\mathcal{M}}

\DeclareMathOperator{\Aut}{Aut}

\DeclareMathOperator{\Ric}{Ric}

\DeclareMathOperator{\Grass}{Grass}
\DeclareMathOperator{\Hilb}{Hilb}
\DeclareMathOperator{\DF}{DF}
\DeclareMathOperator{\Amp}{Amp}
\DeclareMathOperator{\Pic}{Pic}

\DeclareMathOperator{\wt}{wt}

\DeclareMathOperator{\Diff}{Diff}

\DeclareMathOperator{\GL}{GL}

\DeclareMathOperator{\vol}{vol}
\DeclareMathOperator{\ev}{ev}
\DeclareMathOperator{\GW}{GW}

\title[Stable maps in higher dimensions]{Stable maps in higher dimensions}

\author[Ruadha\'i Dervan and Julius Ross]{Ruadha\'i Dervan and Julius Ross}
\address{Ruadha\'i Dervan, DPMMS, Centre for Mathematical Sciences, Wilberforce Road, Cambridge CB3 0WB, United Kingdom}
\email{r.dervan@dpmms.cam.ac.uk}

\address{Julius Ross, Department of Mathematics, Statistics, and Computer Science, University of Illinois at Chicago, 851 S. Morgan Street
Chicago, IL 60607.}
\email{julius@math.uic.edu}

\begin{document}

\begin{abstract} 

We formulate a notion of stability for maps between polarised varieties which generalises Kontsevich's definition when the domain is a curve and Tian-Donaldson's definition of K-stability when the target is a point. We give some examples, such as Kodaira embeddings and fibrations. We prove the existence of a projective moduli space of canonically polarised stable maps, generalising the Kontsevich-Alexeev moduli space of stable maps in dimensions one and two. We also state an analogue of the Yau-Tian-Donaldson conjecture in this setting, relating stability of maps to the existence of certain canonical K\"ahler metrics.

\end{abstract}

\maketitle

\section{Introduction}\label{sec:introduction}
\subsection{Foreword}\label{sec:foreword}

A useful idea in algebraic geometry is that properties of spaces are better phrased as properties of morphisms between spaces.  With this in mind, the goal of this paper is to lay the foundations for what is means for a morphism to be K-stable, as well as to study its basic properties and consequences.   In the absolute case, by which we mean without the morphism, K-stability is a condition that is expected to ensure the moduli space of varieties has good properties (for instance ensuring that it is separated and in some cases proper) and, in the Fano case, is the algebro-geometric condition that is equivalent to the existence of a K\"ahler-Einstein metric through the Yau-Tian-Donaldson correspondence \cite{CDS-jams,GT-97}.

 In particular we shall see:
\begin{enumerate}[(i)]
\item K-stable maps arise naturally in many places, for instance through Kodaira embeddings and in the study of K-stability of fibrations.
\item K-stability for maps appears as a leading order asymptotic in a Geometric Invariant Theory  setup, mimicking the passage from Hilbert stability to K-stability of varieties.
\item In the canonically polarised case, K-stable maps are precisely those with semi-log canonical singularities.  Moreover, there exists a projective moduli space of K-stable maps generalising both the KSBA moduli space of stable pairs, and the moduli space of Konsevich stable maps to higher dimensions.
\item There is a version of the Yau-Tian-Donaldson conjecture in this setting, relating K-stable maps to certain canonical metrics in K\"ahler geometry.
\end{enumerate}

We sketch our definition of K-stability for morphisms, which rests on the notion of a test-configuration.  By
$$p : (X,L)\to (Y,T)$$
we mean a morphism $p:X\to Y$ between varieties $X$ and $Y$ that are endowed with $\mathbb Q$-line bundles $L$ and $T$ respectively.    We will always assume that $X$ is projective, $L$ is ample, and will often require some positivity of $T$ (for instance that it be nef or ample).   A \emph{test-configuration} for $p$ is a projective $\mathbb C^*$-degeneration of this data, by which we mean a projective scheme $\mathcal X$ fitting into a diagram
\[
\begin{tikzcd}
\mathcal X \arrow[swap]{d} \arrow{r}{p}  &Y\times \mathbb P^1  \arrow{dl}  \\
\mathbb P^1 
\end{tikzcd}
\]
along with a line bundle $\mathcal L$ on $\mathcal X$ that is relatively ample over $\mathbb P^1$ and a $\mathbb C^*$-action on $(\mathcal X,\mathcal L)$ covering the usual action on $\mathbb P^1$, and satisfying some additional properties, such that the  fibre of $(\mathcal X,\mathcal L, p)$ over any point in $\mathbb P^1$ other than the fixed point $0\in \mathbb P^1$ is isomorphic to $(X,L,p)$.    

Associated to any test-configuration is a numerical invariant $\DF_p(\X,\L)$, called the \emph{Donaldson-Futaki} invariant.  When the total space $\mathcal X$ is a normal variety, the Donaldson-Futaki invariant is given intersection-theoretically on $\mathcal X$ by the formula
$$\DF_p(\X,\L) = \frac{n}{n+1}\mu_p\L^{n+1} + \L^n.(K_{\X/\pr^1}+p^*T)$$
where $K_{\mathcal X/\mathbb P^1}$ is the relative canonical class, $n=\dim X$ and $\mu_p$ is the constant 
$$\mu_p:=  \frac{-(K_X+p^*T).L^{n-1}}{L^n}$$
to be understood as a ratio of intersection numbers computed on $X$.  There is a notion of the norm $\|(\mathcal X,\L)\|_m$ of a test-configuration, with the property that a test-configuration with zero norm has normalisation that is a trivial product (see Section \ref{sec:kstabilitymaps}).

\begin{definition}[K-stability] We say that $p: (X,L)\to (Y,T)$ is
\begin{enumerate}[(i)]
\item \emph{K-semistable} if  $\DF_p(\X,\L)\ge 0$ for all test-configurations for $p$;
\item \emph{K-stable} if $\DF_p(\X,\L)>0$ for all test-configurations for $p$ with $\|(\X,\L)\|_m>0$;
\item \emph{uniformly K-stable} if there exists an $\epsilon>0$ such that for all test-configurations $(\X,\L)$ for $p$ we have $$\DF_p(\X,\L) \geq \epsilon \|(\X,\L)\|_m.$$\end{enumerate} \end{definition}

 The following are some of the basic properties of K-stable maps we will establish (Theorems \ref{thm:embeddings-body}, \ref{thm:variation} and \ref{thm:factorization}).

\begin{theorem}[Properties of K-stable Maps]\label{thm:props-intro}\ 
Assume that $X$ is semi-log-canonical.  Then
\begin{enumerate}[(i)]\setlength{\itemsep}{2pt} \item(Kodaira Embedding)  Any Kodaira embedding 
$$(X,L)\lhook\joinrel\xrightarrow{L^{\otimes k}}(\pr^{N_k},\scO_{\pr}(1))$$
is a uniformly K-stable map for $k\gg 0$.
\item (Factorisation)  Suppose $(X,L)\to (Y,T)$ factors as $$(X,L)\to (Z,T|_Z)\hookrightarrow (Y,T),$$ for some subvariety $Z$ of $Y$.   Then $(X,L)\to (Z,T|_Z)$ is uniformly K-stable if and only if $(X,L)\to (Y,T)$ is. 
\item (Variation) Uniform K-stability of a map $p:(X,L)\to (Y,T)$ is an open condition as $T$ varies in $\Pic_{\Q}(Y)$.
\end{enumerate}
\end{theorem}

The previous theorem yields many examples of K-stable maps.  It is clear that not all maps are K-stable; for a simple example if $(X,L)$ is a polarised variety that is not K-stable (in the absolute sense) and $(X',L')$ is any other polarised variety then the projection $p: (X\times X', L\boxtimes L')\to (X',L')$ is not K-stable.  
\begin{center}
*
\end{center}

For our next statement, suppose $f:X\to B$ is a fibration such that each fibre of $f$ is smooth and $K_{X/B}$ is relatively ample. For simplicity assume that the fibres also all have trivial automorphism group.   Then $X$ induces a classifying morphism $p: B \to \scM_{can}$ to the moduli space of smooth canonically polarised varieties (again with trivial automorphism group).  Let $V$ be the volume of the fibres, $L_{CM}$ be the CM-line bundle on $\scM_{can}$ (defined in Section \ref{sec:fibrations}) and $L_B$ be a choice of ample line bundle on $B$.
\begin{theorem}[Fibrations]\label{thm:fibrationsintro}\label{thm:introfibration}\
Set $$\delta := \frac{1}{(\dim X-\dim B+1)V}.$$
If the classifying map $$p:(B,L_B) \to (\scM_{can},\delta L_{CM})$$ is not K-semistable, then the polarised variety $$(X,K_{X/B} + mf^*L_B )$$ is not K-semistable (in the absolute sense) for $m\gg 0$.
\end{theorem}
In fact, thinking of a test-configuration for $p:B\to \scM_{gen}$ as a degeneration of $p$ to some map $p_0:\B_0\to \scM_{can}$, a fibre-product construction with the universal family over $\scM_{can}$ gives a test-configuration $\mathcal X$ for $X$ whose limit is the fibration corresponding to $p_0$.  The leading order term in the Donaldson-Futaki invariant for $\mathcal X$ as $m\gg 0$ turns out to be precisely Donaldson-Futaki invariant of the test-configuration for $p:(B,H)\to (\scM_{can},L_{CM})$, which implies Theorem \ref{thm:introfibration}.

The reason for restricting attention to varieties with trivial automorphism group is to ensure the existence of a universal family on the moduli space.  However essentially the same statement holds much more generally, allowing for automorphisms or for fibrations with fibres that are possibly not canonically polarised, as long as these fibres still vary in some kind of moduli space or stack (see Remark \ref{rmk:stacks}).  Another situation in which an analogue of Theorem \ref{thm:fibrationsintro} applies is to the study of K-stability of projective bundles $\pr(E)\to B$, by noting that these are induced by maps (of stacks) from $B$ to the moduli stack of projective space (see Remark \ref{rmk:bundles}).

\begin{center}
*
\end{center}

The definition of K-stability for morphisms extends without difficulty to the log case, in which $X$ is endowed also with a boundary divisor $D$.  In the following assume that $H$ is an ample $\mathbb Q$-line bundle on $Y$.    We shall say that a map  $p: ((X,D);L)\to (Y,H)$ is \emph{canonically polarised} if $L=K_X+D+p^*H$ is ample.

\begin{proposition}[Dervan {\cite[Theorem 1.7, Theorem 1.11]{RD-twisted}}]\label{prop:slcmap}   A canonically polarised map $p: ((X,D);L)\to (Y,H)$ is uniformly K-stable if and only if $(X,D)$ has semi-log canonical singularities (see Remark \ref{rmk:generaltype}).
\end{proposition}

In the following we say that the line bundle $H$ on $Y$ is \emph{sufficiently ample} if there exists an ample line bundle $T'$ on $Y$ with $H-2nT'$ nef. We prove the following:

\begin{theorem}[Moduli Spaces of Canonically Polarised Maps]\label{thm:moduli}\
There exists a separated, projective moduli space $\mathcal M(Y,H)$ of canonically polarised stable maps with fixed target $(Y,H)$ as long as $H$ is sufficiently ample.
\end{theorem}

This recovers Kontsevich's moduli space of stable maps when the domain is a curve, Alexeev's moduli of stable maps where the domain has dimension two \cite{va-stablemaps}, and the KSBA moduli space of stable pairs in the absolute case. We will follow Alexeev's strategy closely, and our result can be seen as an application of the recent progress in the minimal model program.

A natural question is whether one can define higher dimensional Gromov-Witten invariants using such maps, which we shall discuss this in Section \ref{sec:GW}. The main difference in the higher dimensional case is that the analogue of the ``evaluation maps'' in the definition of the usual Gromov-Witten invariants become maps to certain Hilbert schemes, since we are dealing with varieties together with divisors rather than points. The construction of a virtual fundamental class seems out of reach with present techniques, thus as it stands the enumerative invariants we construct are not invariant under deformations of $Y$ in general. Nevertheless, it seems interesting that one can define enumerative invariants using the moduli spaces constructed in Theorem \ref{thm:moduli}.

A special case of Theorem \ref{thm:moduli} recovers the construction of the moduli space of KSBA stable varieties, which gives a moduli space of birational models of varieties of general type. Conjecturally, when $X$ is not uniruled, the minimal model program terminates with an Iitaka fibration $X\to X^{can}$, such that the generic fibre of the fibration is Calabi-Yau, and $K_{X^{can}}+D$ is ample, for some divisor $D$ encoding the multiple fibres of the fibration. When \emph{all} fibres are Calabi-Yau, this implies that the map $X^{can}\to (\M_{CY},L_{CM})$ is a stable map, where $(\M_{CY},L_{CM})$ is a moduli space of Calabi-Yau varieties endowed with the CM line bundle. In a formal sense, a similar result holds in general replacing $L_{CM}$ with the moduli part of the fibration (see for example \cite{FA}). This suggests a link between K-stability of maps and the minimal model program for generalised polarised pairs, which has been a crucial ingredient in the recent progress in understanding of the minimal model program for varieties which are not of general type \cite{CB,birkar2}. Two questions arise naturally from these observations: firstly, whether or not one can form a birational moduli space of non-uniruled varieties using K-stable maps, and secondly, whether or not one can give a K-stability interpretation of the remaining case of Mori fibre spaces, which should include K-stability of Fano varieties as a special case.

\begin{center}
*
\end{center}

A primary motivation for considering K-stability in the absolute case is the link with canonical K\"ahler metrics.  To discuss the analogy for maps, suppose $\alpha \in c_1(T)$ is a K\"ahler metric on $Y$. We say that a $(1,1)$-form $\omega \in c_1(L)$ on $X$ is \emph{$p^*\alpha$-twisted constant scalar curvature K\"ahler} (or simply \emph{twisted cscK}) if $\omega$ is positive, so is a K\"ahler form, and satisfies
$$ \operatorname{Scal}(\omega) - \Lambda_{\omega} p^*\alpha \equiv const.$$
where $\operatorname{Scal}(\omega)$ denotes the scalar curvature of the Riemannian metric associated to $\omega$.  The following is the analogy of the Yau-Tian-Donaldson conjecture for maps.

\begin{conjecture}\label{YTD}Suppose $p: (X,L)\to (Y,T)$ has discrete automorphism group (which will occur, for instance, if $p^*T$ is ample).  Then $c_1(L)$ admits an $p^*\alpha$-twisted cscK metric if and only if the map $p:(X,L)\to (Y,T)$ is uniformly K-stable.
\end{conjecture}

The analogous conjecture when the map $p$ has automorphisms is more subtle; in this case the most likely candidate for the stability notion is some form of uniform or filtration K-polystability.

In the Fano case, which means that $L=-K_X-p^*T$ is ample, it should even be true that the existence of a $p^*\alpha$-twisted cscK metric is equivalent to $K$-stability of $p$ (possibly allowing $p$ to develop some singularities in the central fibre, see Remark \ref{rmk:twistedKstability}).  We expect that this can be proved by following the classical continuity approach of Aubin \cite{Aubin-reduction} executed by Datar-Sz\'ekelyhidi \cite{SzekelyhidiDatar-Kahler}, and plan to return to this topic in the future.

The knowledgeable reader will notice a similarity between our definition of K-stability for maps and the definition of twisted K-stability \cite{RD-twisted}.  But we emphasise that whereas it is the case that $DF_p$ is precisely the twisted Donaldson-Futaki invariant with respect to $p^*T$, in our definition of K-stability we are demanding that the total space of the test-configuration comes with a morphism to $Y$ extending $p$ (the reader is referred to Remark \ref{rmk:twistedKstability} for further discussion on this point).    Using this twisted point of view, we can observe that it is already known that the easy direction of Conjecture \label{YTD}  has almost been proved in \cite[Theorem 1.1]{RD-twisted}, namely that the existence of such a metric implies K-semistability (and even uniform K-stability if $p^*\alpha$ is positive).

Through the link with twisted cscK metrics, one also obtains a stability interpretation of the continuity method. For instance, when $X$ is a Fano manifold with $\alpha\in c_1(X)$ arbitrary, following Aubin's continuity method Sz\'ekelyhidi \cite{GS-lower} defines the greatest lower bound on the Ricci curvature as \begin{align*} R(X) &:= \sup \{0<t<1: \exists \ \omega_t \in c_1(X) \textrm{ with } \Ric{\omega_t} > t\omega_t \}, \\ &= \sup \{0<t<1:  \exists \ \omega_t \in c_1(X) \textrm{ with } \Ric{\omega_t} = t\omega + (1-t)\alpha \}.\end{align*}  It follows from \cite{RD-twisted, GS-lower} that$$R(X) = \sup_{\{k\in \N \textrm{ and } t \in (0,1]\}}\left\{ p: (X,-K_X) \lhook\joinrel\xrightarrow{-kK_X} \left(\pr^{N_k}, \scO_{\pr}\left(\frac{1-t}{k}\right)\right) \textrm{ is stable} \right\}.$$ Similarly, for a general polarised variety $(X,L)$, stability of the Kodaira embeddings using $L^k$ are related to the continuity method for cscK metrics discussed by Chen \cite{XC-continuity} and the analogue of $R(X)$ discussed by Chen \cite[Definition 1.16]{XC-continuity} and Hashimoto \cite[Section 1.3]{YH}.

\subsection{Comparison with other work}
K-stability, in the absolute case, was introduced by Tian \cite{Tian-hypersurface,GT-97} in the Fano case building on the invariant of Futaki \cite{Fut-obs}, and was later generalised to include more general polarisations and singularities by Donaldson \cite{SD-toric}.  The intersection-theoretic approach that we take is due to Odaka \cite{odaka-disc} and Wang \cite{XW}.

While the notion of a K-stable map we introduce is new, the analytic counterpart of twisted cscK metrics has appeared in previous work. Apart from their use in Aubin's continuity method \cite{Aubin-reduction}, these metrics originally appeared in the work of Fine \cite{JF} and Song-Tian \cite{ST}. Theorem \ref{thm:fibrationsintro} is an algebraic converse to the result of Fine, who proves that if $(B,L_B)$ admits a unique twisted cscK metric,  where the twisting is determined by the fibration,  and the fibres $(X_b,L_{X_b})$ all admit a unique cscK metric, then $(X,mf^*L_B + L_X)$ admits a cscK metric for all $m \gg 0$. 

Song-Tian prove that under certain regularity hypotheses, if $K_X$ is semiample, then the resulting Iitaka fibration $X\to X_{can}$ is realised by the unnormalised K\"ahler-Ricci flow on $X$. That is, if  $\omega_t$ is a family of K\"ahler metrics induced from the K\"ahler-Ricci flow, one has Gromov-Hausdorff convergence $(X,\omega_t) \to (X_{can},\omega_{can})$, where $\omega_{can}$ is a twisted K\"ahler-Einstein metric with respect to the pullback of a Weil-Petersson type metric on the moduli space of Calabi-Yau varieties, and with certain cone singularities along the divisors encoding the multiple fibres of the fibration. This fits in well with what we mention above, that the end product of the minimal model program for non-uniruled varieties has a natural stability interpretation using stable maps.

The analytic analogue of Theorem \ref{thm:props-intro} $(i)$ is due to Hashimoto and Zeng independently \cite[Theorem 1.2]{YH}\cite[Theorem 1.1]{YZ}, who prove the existence of twisted cscK metrics with a large twist. Hashimoto also proves a more general result using critical points of the J-flow, which is the counterpart to our more general Theorem \ref{thm:j-stab}, which uses J-stability.  In Theorem \ref{thm:j-stab}, using J-stability we also prove an algebraic converse to Hashimoto's result. We remark that in Theorem \ref{thm:uniform-kodaira} we are also able to give a uniform $k$ for flat families of varieties for which the Kodaira embedding is a stable map, provided one embeds using certain adjoint bundles. It would be very interesting, but challenging, to prove a similar statement for twisted cscK metrics.

The first stability notion for twisted cscK metrics is due to Stoppa \cite{JS-twisted}, who proves that the existence of such a metric implies a twisted version of slope semistability. Stoppa also provides a moment map interpretation for the existence of twisted cscK metrics \cite[Section 2]{JS-twisted}; it would be interesting to give a moment map interpretation for twisted cscK metrics which is closer to the point of view of stable maps. 

Proposition \ref{prop:slcmap} (from \cite{RD-twisted}) is proved in a similar way to Odaka's technique in the absolute case \cite{odaka-calabi,odaka-disc}. 
 
\subsection{Acknowledgements} The authors would like to thank Giulio Codogni, Kento Fujita and Jacopo Stoppa and Gabor Sz\'ekelyhidi for helpful discussions. The first author especially thanks Roberto Svaldi and Chenyang Xu for birational advice. We also thank the referee for their comments.

\subsection{Notation}
We work throughout over the complex numbers. We often mix multiplicative and additive notation for line bundles, especially when computing intersection numbers. We often have a map of varieties or schemes $p: (X,L)\to (Y,T)$, which we shall abbreviate to $p$. When one has a natural map $f: Z\to W$ of varieties or schemes, the pullback of  a line bundle $H$ on $W$ to $Z$ is often simply written as $H$.

\section{Stability notions for maps}\label{sec:definitions}

\subsection{K-stability for maps}\label{sec:kstabilitymaps}

Suppose that $X,Y$ are schemes with $X$ projective and equidimensional of dimension $n$, that $L,T$ are line bundles on $X$ and $Y$ respectively with $L$ ample, and $p:X\to Y$ is a morphism.  We shall write this data as
$$p:(X,L)\to (Y,T).$$

\begin{definition}(Twisted-slope) For a $\mathbb Q$-line bundle $T'$ on $X$ the \emph{twisted slope} is defined to be $$\mu_{T'}(X,L) = \frac{-(K_X+T').L^{n-1}}{L^n}$$
where $K_X$ denotes the degree two part of the singular Todd class as defined by Fulton \cite[p354]{fulton-book}. In particular if $X$ is a normal variety then $K_X$ can be replaced with the canonical Weil-divisor.   When $T'=0$ we abbreviate this to $\mu(X,L)$, and given a morphism $p:(X,L)\to (Y,T)$ we abbreviate this to 
$$\mu_p = \mu_p(X,L) = \mu_{p^*T}(X,L) = -\frac{(K_X+p^*T).L^{n-1} }{L^n}.$$
\end{definition}

\begin{remark}
We warn the reader that our convention differs from the twisted slope in \cite[p4735]{RD-twisted} in which $T'$ is replaced by $2T'$, and by the definition of the slope in \cite[Equation (1.5)]{RT} by a factor of $n/2$.
\end{remark}

\begin{definition}\label{def:testconfiguration}
A \emph{test-configuration} for  $p: (X,L)\to (Y,T)$ is a scheme $\X$ together with 
\begin{enumerate}[(i)]
\item a flat map $\pi: \X \to \pr^1$,
\item a relatively ample line bundle $\L \to \X$,
\item a $\C^*$-action on $\X$ lifting to $\L$ and covering the natural action on $\pr^1$,
\item an equivariant map $p:\X\to Y$ extending the map $p:X\to Y$ on the general fibre, with $Y$ given the trivial action,
\item a $\C^*$-equivariant isomorphism $\pi^{-1}(\pr^1 \backslash 0) \cong (X\times \C,L^r)$, where $\C^*$-acts trivially on $X$ and in the natural way on $\C$. 
\end{enumerate} We call $r$ the \emph{exponent} of the test-configuration. For a \emph{semi-test-configuration} we merely require that $\L$ is relatively \emph{semi}-ample.  Abusing notation, we shall denote a test-configuration by $p:(\mathcal X,\mathcal L)\to (Y,T)$.
\end{definition}

\begin{definition}[Donaldson-Futaki invariant] Let $p: (\X,\mathcal L)\to (Y,T)$ be a test-configuration of exponent $r$. Then the \emph{Donaldson-Futaki} invariant is defined to be
\begin{equation}\DF_p(\X,\L) = \frac{n}{n+1}\mu_p(X,L^r)\L^{n+1} + \L^n.(K_{\X/\pr^1} + T)\label{eq:DFintersection}
\end{equation}
where $K_{\X}$ is the degree two part of the singular Todd class of $\X$ \cite[p354]{fulton-book}.  Again, if $\mathcal X$ is a normal variety then it can be replaced with the canonical Weil-divisor.\end{definition}

\begin{remark}
The Donaldson-Futaki invariant can be defined in at least two other ways.  One way is to use asymptotics of Hilbert and weight polynomials, as we will discuss below in Section \ref{sec:GIT}.  Another, used by many including the authors \cite[Section 2.1]{Dervan-Ross}, is to consider only normal varieties $X$ and to require that the total space $\mathcal X$ of a test-configuration also be normal to ensure the relative canonical divisor $K_{\mathcal X/\mathbb P^1}$ exists (say as a Weil-divisor) so that the intersection-theoretic quantity on the right-hand-side of \eqref{eq:DFintersection} is defined. 
\end{remark}

Before defining K-stability, we need to know what it means for a test-configuration to be trivial. Fix an (equivariant) resolution of indeterminacy of the natural birational map $f:(X\times\pr^1,L)\dashrightarrow (\X,\scL)$ as follows.
\begin{equation}\label{resolution}
\begin{tikzcd}
\Y \arrow[swap]{d}{q} \arrow{dr}{} &  \\
X\times\pr^1 \arrow[dotted]{r}{} & \X
\end{tikzcd}
\end{equation}

\begin{definition}[Dervan] \cite[Definition 2.5, Remark 3.11]{RD-twisted}\label{lem:minimum} The {minimum norm} is defined to be $$\|(\X,\L)\|_m = \L^n.q^*L - \frac{nr^{-1}}{n+1}\L^{n+1},$$
which is easily checked to be independent of choice of resolution. 
\end{definition}

\begin{remark}The minimum norm was independently introduced by Boucksom-Hisamoto-Jonsson \cite[Definition 7.6]{BHJ}, who called it the ``non-Archimedean $(I-J)$ functional''.  The same quantity also appeared in the work of Lejmi-Sz\'ekelyhidi in a somewhat different situation \cite[p416]{Lejmi-Sz}. The definition is motivated by an analogous functional on the space of K\"ahler metrics. However, one can also show that it is Lipschitz equivalent to the $L^1$-norm introduced by Donaldson \cite[p470]{SD-lower}, see \cite[Theorem 7.9]{BHJ} or \cite[Theorem 1.5]{RD-relative} for an analytic proof when $X$ is a smooth variety. 
\end{remark}

The justification for the use of the word ``norm'' is as follows. Note that the normalisation of a test-configuration for $X$ is a test-configuration for the normalisation of $X$.

\begin{lemma} Assume $X$ is a variety.  A test-configuration normalises to the trivial test-configuration if and only if its minimum norm is zero. \end{lemma}

\begin{proof} This is proven in both \cite[Theorem 4.7]{RD-twisted} and \cite[Corollary B]{BHJ} when $X$ itself is normal. But the minimum norm is preserved by normalisation, as is clear from its intersection theoretic definition, hence we can reduce to the normal case.\end{proof}

With this in place, we can define K-stability for maps. 

\begin{definition}[K-stability] We say that $p: (X,L)\to (Y,T)$ is
\begin{enumerate}[(i)]
\item \emph{K-semistable} if $\DF_p(\X,\L)\ge 0$ for all test-configurations for $p$;
\item \emph{K-stable} if $\DF_p(\X,\L)>0$ for all test-configurations for $p$ with $\|(\X,\L)\|_m>0$;
\item \emph{uniformly K-stable} if there exists an $\epsilon>0$ such that for all test-configurations $(\X,\L)$ for $p$ we have $$\DF_p(\X,\L) \geq \epsilon \|(\X,\L)\|_m.$$\end{enumerate} \end{definition}

\begin{remark}[Scaling]\label{rmk:scaling}
A map $p:(X,L)\to (Y,T)$ is K-semistable if and only if for any $m>0$ the map $p:(X,L^{\otimes m}) \to (Y,T)$ is K-semistable (with similar statements for K-stability and uniform K-stability).  For this reason one can allow $L$ to be a $\mathbb Q$-line bundle. The definition also extends to the case when $T$ is a $\Q$-line bundle directly.
\end{remark}

\begin{remark}[Absolute case]\label{rmk:absolute}If $Y$ is a single point then $p:(X,L)\to (Y,T)$ is K-semistable if and only if $(X,L)$ is K-semistable in the usual sense (say as defined by Donaldson \cite[Definition 2.12]{SD-toric}).
\end{remark}

\begin{remark}[Twisted K-stability]\label{rmk:twistedKstability} The quantity $DF_p(\mathcal X,L)$ is precisely the twisted Donaldson-Futaki invariant with respect to the line bundle $p^*T$ 
as considered in \cite{RD-twisted,JS-twisted}. Thus if $(X,L)$ is twisted K-semistable  (resp.\ K-stable or uniformly K-stable) with respect to $p^*T$ then the same is true of $p:(X,L)\to (Y,T)$.   The converse, however, is not immediate since in principle there could be test-configurations for $(X,L)$ that do not come with a morphism to $Y$ extending $p$. However, as we will see in Section \ref{sec:fact}, this is not an issue for uniform K-stability, and thus $(X,L)$ is uniformly twisted K-stable with respect to $p^*T$ if and only if $p$ is uniformly K-stable.  To fix ideas we will in this paper only consider this stronger notion of K-stability (in which we require the map to extend to the central fibre).  However it may turn out in the future that one wants to relax this slightly (say by allowing the function to obtain some singularities) which may well occur when one wishes to relate this with canonical K\"ahler metrics.


It is worth pointing out that in \cite{RD-twisted} (which was motivated by \cite{Lejmi-Sz}), the corresponding notion of ``twisted Hilbert stability'' does not appear to be a genuine GIT notion, instead it arises as a sort of stability notion for the linear system $|p^*T|$. Somewhat surprisingly however, the resulting numerical invariant is equal to the one that appears in the GIT setup we will consider below.   This notion of Hilbert stability for linear systems is described in more detail in \cite[Section 4]{Dervan-Keller}.

\end{remark}

\begin{remark}\label{filtrations}It is clear that $$\textrm{uniformly K-stable} \Rightarrow \textrm{K-stable}\Rightarrow \textrm{K-semistable}.$$ Even when $Y$ is a point, there are examples of K-semistable varieties which are not K-stable. We expect that K-stability is not equivalent to uniform K-stability, though no such example is known. Moreover, uniform K-stability should be the correct moduli notion for higher dimensional maps, at least for forming \emph{separated} moduli. Moreover, we expect that uniform K-stability is a Zariski open condition in flat families of maps, which is an essential condition for forming moduli spaces. Forming a \emph{proper} moduli space appears to be much more subtle.

A variant of these notions would be a notion of \emph{filtration} K-stability for maps, extending the definition of Sz\'ekelyhidi in the absolute case \cite[Definition 4]{sz-filtrations}. However in the absolute case, it is not hard to see that uniform K-stability implies filtration K-stability. Since this is only a condition on the norms, this would presumably hold in the case of maps as well. Thus the various examples of uniformly K-stable maps we provide would likely also be examples of filtration K-stable maps. \end{remark}

The extension to pairs is as follows.  Given a $\Q$-Weil divisor $D$ on $X$, we will denote by $\D$ the closure $\overline{\C^*.D}\subset\X$, which is a $\Q$-Weil divisor.
\begin{definition}[Log K-stability] \label{def:logKstability} The \emph{log Donaldson-Futaki invariant} of $p: ((\X,\D),\L)\to (Y,T)$ is the intersection number $$\DF_p((\X,\D);\L) = \frac{n}{n+1}\mu_{D+p^*T}(X,L)\L^{n+1} + \L^n.(K_{\X/\pr^1}+\D+p^*T).$$ 
We then define log K-semistability, log K-stability and uniform log-K-stability exactly as before. 
\end{definition}
In the absolute case, log K-stability is due to Donaldson \cite[Section 6]{SD-cone}. 

\begin{remark}[Singularities]\label{rmk:singularities}
Using the relationship with twisted K-stability, we know from \cite[Theorem 3.28]{RD-twisted} that if $p:((X,D),L)\to (Y,T)$ is $K$-semistable then $(X,D)$ has semi-log canonical singularities (see Definitions \ref{def:lc} and \ref{def:slc}). 
\end{remark}

\begin{remark}[Calabi-Yau and General Type maps]\label{rmk:generaltype}
 Again using the terminology of Definitions \ref{def:lc} and \ref{def:slc}, from the relationship with twisted K-stability, \cite[Theorem 1.2, Theorem 3.28]{RD-twisted} tells us that if $p:((X,D),L)\to (Y,T)$ is a Calabi-Yau map (by which we mean $K_X+p^*T$ is numerically trivial, with $L$ arbitrary) then $p$ is K-stable (or uniformly K-stable) if and only if $(X,D)$ has Kawamata log terminal singularities.  Moreover if $p:((X,D);L))\to (Y,T)$ is is canonically polarised (by which we mean $L=K_X+D+p^*T$ is ample) then $p$ is K-stable (or uniformly K-stable, K-semistable) if and only if $(X,D)$ has semi-log canonical singularities.
\end{remark}

\subsection{Hilbert Stability for maps}\label{sec:GIT}

We start by recalling the Hilbert-Mumford criterion from Geometric Invariant Theory (GIT).   Suppose $Z\subset \pr^N$ is a projective scheme that is invariant under the action of a reductive subgroup $G\subset SL(N+1)$.  Fixing $z\in Z$, for each one-parameter subgroup $\lambda \hookrightarrow G$ there is a limit $z_0:=\lim_{t\to 0}\lambda(t).z$ in $Z$.  Then $z_0$ is fixed by $\lambda$, so picking a non-zero point $\hat z_0\in \C^{N+1}$ lying on the line above $z_0$, the one-parameter subgroup $\lambda$ acts as $$\lambda(t).\hat{z}_0 = t^{-\mu(z,\lambda)}\hat{z}_0$$ for some integer $\mu(z,\lambda)$ called the \emph{Mumford weight}.    When $\lambda\hookrightarrow GL(N+1)$ we compose with an action that scales $\mathbb C^{N+1}$ with some weight to obtain a  rational one-parameter subgroup $\hat{\lambda}\hookrightarrow SL(N+1)$ that has the same action on $Z$, and set
\begin{equation}\label{eq:HMGL}
\mu(\lambda,z) = \mu(\hat{\lambda},z).
\end{equation}
The Hilbert-Mumford numerical criterion \cite[Theorem 2.1]{GIT} states that the point $z\in Z$ is
\begin{enumerate}[(i)]
\item \emph{semistable}  if $\mu(z,\lambda)\geq 0$ for all one-parameter subgroups $\lambda$,
\item \emph{stable}   if $\mu(z,\lambda)> 0$ for all non-trivial one-parameter subgroups $\lambda$,
\item \emph{polystable}  if $\mu(z,\lambda)\geq 0$ for all one-parameter subgroups $\lambda$, with equality if and only if $\lambda$ fixes $z$.
\end{enumerate}
Roughly speaking, the GIT quotient $Z\sslash G$ parameterises polystable orbits. 
 
We apply this now to define Hilbert stability for maps.  Let $p:(X,L)\to (Y,M)$ be a morphism between polarised varieties and assume that $M$ is very ample.  Let $n=\dim X$, and $r$ be large enough so that $L^r$ is very ample.   Let $V_r$ be a fixed vector space of dimension $h^0(X,L^r)$, let $W=H^0(Y,M)$  and fix an isomorphism $H^0(X,L^r)\simeq V_r$.    Then the graph $\Gamma_p$ of $p$ is a subscheme
$$\Gamma_p \subset X\times Y \subset \pr(V_r)\times\pr(W) \subset \pr(V_r\otimes W).$$   
Thus $\Gamma_p$ defines a point  in an appropriate Hilbert scheme $$[\Gamma_p]\in \Hilb = \Hilb(\mathbb P(V_r\otimes W))$$ of $\pr(V_r\otimes W)$.    There is ambiguity given by possibly different choice of isomorphism $H^0(X,L^r)\simeq V_r$, and so we are interested in the orbit of $[\Gamma_p]$ under the natural $\GL(V_r)$ action. (Observe that as we are thinking of $Y$ as fixed, there is no additional action coming from the automorphism group of $W$).  

\begin{remark}
When $\dim X=1$ this is the setup considered by Baldwin-Swinarski  \cite[Section 3]{BS} who use Geometric Invariant Theory to produce the moduli space of stable maps from curves.
\end{remark}

We recall briefly the construction of the Hilbert scheme.  For sufficiently large integer $k$ consider the exact sequence $$ 0\to H^0(\pr(V_r\otimes W),\scI_{\Gamma_p}(k)) \to S^k (V_r\otimes W) \to H^0(\Gamma_p, \scO_{\mathbb P(V_r\otimes W)}(k)|_{\Gamma_p})\to 0$$ where $\scI_{\Gamma_p}$ is ideal sheaf defining $\Gamma_p$.    We think of this as a point in the Grassmannian of the fixed vector space $S^k(V_r\otimes W)$, and $\Hilb(V_r\otimes W)$ as the locus inside this Grassmannian that parameterises such subschemes.  So by the Pl\"ucker embedding
\begin{equation}[\Gamma_p]\in \Hilb(V_r\otimes W) \subset \Grass(S^k(V_r\otimes W))\subset \mathbb P(\Lambda^h S^k (V_r\otimes W))=:\mathbb P\label{eq:hilbertscheme}\end{equation}
where $h:=\dim H^0(\pr(V_r\otimes W),\scI_{\Gamma_p}(k))$.  Clearly the natural action of $GL(V_r)$ on $\mathbb P$ preserves $\Hilb(V_r\otimes W)$, and thus we are in precisely the setup of the Hilbert-Mumford criterion described above.

\begin{definition}[Hilbert-stability of a map]
We say that the map $p$ is \emph{Hilbert-stable at level $r$} if for all sufficiently large $k$, the point $[\Gamma_p]$ in \eqref{eq:hilbertscheme} is stable with respect to the action of $GL(V_r)$.  We say $p$ is \emph{asymptotically Hilbert-stable} if it is Hilbert-stable at level $r$ for all $r$ sufficiently large. 
\emph{(Asymptotic) Hilbert-semistability} and \emph{Hilbert-polystability} of $p$ are defined similarly.  
\end{definition}

\subsection{Donaldson-Futaki invariant as a leading term of a Mumford-weight}
Suppose now that $p:(\mathcal X,\mathcal L)\to (Y,M)$ is a test-configuration for $p$.  It is not hard to see that the restriction of the test-configuration over $\mathbb C\subset \mathbb P^1$ gives rise to a one-parameter subgroup $\lambda\hookrightarrow GL(V_r)$ whose limit $\Gamma_{p_0} = \lim_{t\to 0} \lambda(t).\Gamma_p$ is the graph of the map $p:\mathcal X_0\to Y$, and thus $\Gamma_{p_0}$ is equivariantly isomorphic to $\mathcal X_0$.  (Conversely such a one-parameter subgroup clearly defines a test-configuration, see for example \cite[Proposition 3.7]{RT} in the absolute case $Y=\{pt\}$).

To discuss this further it is useful to have some notation.   So set \begin{align*} h(r,k) &= \dim H^0(X, L^{rk}\otimes M^{k}) = \dim H^0(\X_0,\L_0^k\otimes M^k), \\  w(r,k) &= \wt(H^0(\X_0,\L_0^k\otimes M^k)), \\  \hat{h}(r) &= \dim H^0(X, L^{r})= \dim H^0(\X_0, \L_0^{r}), \\  \hat w(r) &= \wt(H^0(\X_0,\L_0^r)), \end{align*} where the equality of the Hilbert polynomials arises from flatness and $\wt(H^0(\X_0,\L_0^k\otimes H^k)$ denotes the total weight of the $\C^*$-action on $H^0(\X_0,\L_0^k\otimes M^k)$ induced from the action on $(\X_0,\L_0\otimes M)$.  Also set

$$\tilde{w}(r,k):= w(r,k)\hat  h(r) - \hat w(r)kh(r,k)$$

\begin{lemma} 
The Mumford-weight of the one-parameter subgroup $\lambda\hookrightarrow GL(V_r)$ induced by the test-configuration $(\mathcal X,\mathcal L)\to (Y,M)$ is given by
$$ \mu([\Gamma_p],\lambda) = \frac{\tilde{w}(r,k)}{\hat{h}(r)}.$$
\end{lemma} 
\begin{proof}
By the definition of the Pl\"ucker embedding used in \eqref{eq:hilbertscheme}, the line in $\Lambda^h S^k(V_r\otimes W)$ over the point $[\Gamma_{p_0}]$ is naturally isomorphic to
$$\Lambda:= \Lambda^{max} H^0(\Gamma_{p_0},\mathcal O_{\mathbb P(V_r\otimes W)}(k)|_{\Gamma_{p_0}}) \otimes  \Lambda^{max} S^k(V_r\otimes W)^*.$$
Under the idenfitication between $\Gamma_{p_0}$ and $\X_0$ the line bundle $\mathcal O_{\mathbb P(V_r\otimes W)}(1)$ pulls back to $L^r\otimes p^*M$.  Hence
\begin{equation}\Lambda \simeq \Lambda^{max} H^0(\X_0,\L_0^{k}\otimes p^* M^k) \otimes  \Lambda^{max} S^k(V_r\otimes W)^*\label{eq:idenfiticationHMbundle}.\end{equation}
We can compose $\lambda$ with the action that scales the vector space $V_r$ by some weight $\alpha \in\Q$ to get a new rational one-parameter subgroup $\hat{\lambda}\hookrightarrow SL(V_r)$. But the induced weight of $\hat{\lambda}$ on $V_r$ is precisely $\hat{w}(r) + \alpha \hat{h}(r)$, and thus to ensure $ \hat{w}(r) + \alpha \hat{h}(r) =0$ we must have
$$ \alpha := -\frac{\hat{w}(r)}{\hat{h}(r)}.$$

Observe that this in particular implies that $\hat{\lambda}$ factors through both $SL(\Lambda^h S^k(V_r\otimes W))$ and moreover acts with zero weight on $\Lambda^{max} S^k(V_r\otimes W)$.  Hence using \eqref{eq:HMGL} and then \eqref{eq:idenfiticationHMbundle}
$$\mu(\lambda,[\gamma]) = \mu(\hat{\lambda},[\gamma]) = w(r,k) + \alpha k h(r,k) = \frac{\tilde{w}(r,k)}{\hat{h}(r)},$$ 
where the factor of $k$ arises as if one adds a constant $c$ to the weights of a $\C^*$-action on a general $(Z,L_Z)$, this changes the total weight on $H^0(Z,L_Z^l)$ by $cl$. The result follows.\end{proof}

All of the above quantities are polynomials in $r$ and $k$.  Thus we can write
\begin{equation}\label{eq-chow}\tilde w(r,k) = \sum_{i=1}^{n+1} e_{i}(r)k^i.\end{equation}

It is convenient now to set $M = H^{\otimes 2}$ and consider the map $q: (X,L) \to (Y,H)$ induced from $p: (X,L) \to (Y,M)$ (the factor of two is to make the formulae cleaner).

\begin{proposition}[$DF_p$ as the leading term of a Mumford-weight]\label{prop:DFasleading}
There is an expansion
\begin{equation} e_{n+1}(r) = \DF_{q}(\mathcal X,\mathcal L) r^{2n} + O(r^{2n-1}).\label{eq:DFasleading}
\end{equation}
\end{proposition}

It is convenient to start with a preliminary statement:

\begin{lemma}\label{lem-euler} For $r,k \gg 0$ have \begin{align*} h(r,k) &= \chi(X,L^{rk}\otimes M^k), \\  w(mr,k) &= \chi(\X,\L^{mk}\otimes M^k) - \chi(X,L^{mrk}\otimes M^k), \\  \hat{h}(r) &= \chi(X, L^{r}), \\  \hat w(mr) &=\chi(\X,\L^m) -  \chi(X, L^{mr}), \end{align*} where the Euler characteristics are computed on the compactified family $\X\to\pr^1$. 
\end{lemma}

\begin{proof} For the dimensions of the vector spaces, this follows from Riemann-Roch. For the weights, this is  due in various forms to the first author, Donaldson, Odaka and Wang \cite{RD-twisted,SD-lower,odaka-RT,XW}.\end{proof}

\begin{proof}[Proof of Proposition \ref{prop:DFasleading}]
It is clear that $r\mapsto e_{n+1}(r)$ is a polynomial in $r$ of degree at most $2n+1$.  We claim now that the degree $2n+1$ term actually vanishes.  It is clearer to work with $w(mr,k)$, considering $r$ fixed and varying $m$ and $k$. First note that \begin{align*}\tilde w(mr,k) &= w(mr,k)\hat  h(mr) - \hat w(mr)kh(mr,k), \\ &= ( \chi(\X,\L^{mk}\otimes M^k) - \chi(X,L^{mrk}\otimes M^k)) \chi(X, L^{mr})  \\ & \ \ \ \ \ \ \ \ \ \ \ \ \ \  \ \ \ -(\chi(\X,\L^m) -  \chi(X, L^{mr}))\chi(X,L^{mrk}\otimes M^k) \end{align*}by Lemma \ref{lem-euler}. Thus by asymptotic Riemann-Roch and using additive notation $$e_{n+1}(mr) = \frac{(m\L+ M)^{n+1}}{(n+1)!}\chi(X, L^{mr}) - (\chi(\X,\L^{mr}) -   \chi(X, L^{mr}))\frac{(mrL+M)^{n}}{n!}.$$ From this one sees the degree $2n+1$ term in $m$ vanishes, as claimed.

To make the calculation more explicit, and using the variable $l$ for clarity, denote \begin{align*} \hat{h}(l) &= \dim H^0(X,L^l) =  a_0l^n + \hat a_1 l^{n-1} + O(l^{n-2}), \\  \hat w(l) &= \wt(H^0(\X_0,\L_0^l))=  b_0 l^{n+1} + \hat b_1 l^n + O(l^{n-1}). \end{align*} Define also $$ a_q = \frac{L^{n-1}.H}{(n-1)!}, \quad   b_q = \frac{\L^n.H}{n!},$$ which arise in a simple manner from the Hilbert and weight polynomials defining $\tilde w(r,k)$. It is now clear that $e_{n+1}(r)$ admits an expansion $$ e_{n+1}(r) = \frac{ b_0 ( a_1+  a_q)  - ( b_1 +  b_q) a_0}{ a_0} r^{2n} + O(r^{2n-1}),$$ thus it suffices to show that this equals the Donaldson-Futaki invariant of $q: (\X,\L) \to (Y,H)$. But this is immediate by Fulton's asymptotic Riemann-Roch for singular schemes \cite[Corollary 18.3.1 (a)]{fulton-book}, which states that $$\chi(\X,\L^l) = \frac{\L^{n+1}}{(n+1)!}l^{n+1} + \frac{\L^n.K_{\X}}{2n!}l^n + O(l^{n-1}),$$ where $K_{\X}$ is the degree two part of the (singular) Todd class as usual, together with $\L^n.K_{\pr^1} = -2L^n$.

\end{proof}

\begin{remark}The term $e_{n+1}(r)$ governs Chow stability of the map $p$, where one uses GIT with respect to the Chow line bundle on the Hilbert scheme.  This is precisely as in the absolute case, see for example \cite[Theorem 3.9]{RT}.   Thus one can think of the Donaldson-Futaki invariant as the leading order term in asymptotic Chow stability.
\end{remark}

\begin{remark}Equation \eqref{eq:DFasleading} can also be used as the Donaldson-Futaki invariant (thus avoiding  intersection theory).  This is the approach taken by Donaldson \cite[Definition 2.12]{SD-toric} who was the first to define this invariant (in the absolute case) in this level of generality.
\end{remark}

There is an analogous calculation when one has a divisor $D\subset X$, which leads to the log Donaldson-Futaki invariant being the leading order term of certain Mumford-weights, computed on a product of Hilbert schemes (see for example \cite[Theorem 4.9]{Dervan-Keller}). An interesting point is that this relies on $D$ being a divisor; if $Z$ has codimension at least two, then the leading order term of the Mumford weight is unaffected for dimensional reasons. Thus while it makes sense to ask for a map $p: ((X,Z);L) \to (Y,H)$ to be asymptotically Hilbert stable for $Z$ an arbitrary subscheme of $X$, the same question for K-stability  does not unless $Z$ is a divisor.

For completeness, we show that the minimum norm can also be defined in this way.  Choose $r$ such that $L^r$ is very ample. For each $D\in |L^r|$, and each test-configuration $p:\X\to D$ we obtain a test-configuration for $(D,L)$ by setting the total space $\D$ to be the closure of the orbit of $D$, i.e. $D = \overline{\C^*.D}$. As above, we obtain a polynomial  for $k \gg 0$ $$\wt(H^0(\D_0,\L_0^l)) = b_{0,D}l^n+O(l^{n-1}).$$ One can show that $b_{0,D}$ is in fact constant outside a Zariski closed subset of $D$ \cite[Lemma 9]{Lejmi-Sz}, so we set $\tilde b_0$ to equal this general value. 

\begin{lemma}[Minimum norm]\cite[Remark 3.11]{RD-twisted} The minimum norm of $p: (\X,\L)\to (Y,H)$ is given by $$\|(\X,\L)\|_m = \frac{n!}{r}\left (\tilde b_0 - nb_0\right ).$$\end{lemma}

The minimum norm also arises as a limit of certain finite dimensional Mumford weights \cite[Theorem 4.9]{Dervan-Keller}.

\section{Stability criteria}

This section provides several situations in which one can show a map is uniformly K-stable, proving Theorem \ref{thm:props-intro}. These results apply when $X$ has log canonical singularities, and in particular hold when $X$ is smooth. It is straightforward to extend the results to pairs. For ease of exposition, we shall only consider the log canonical case in this section, so in particular throughout we assume $X$ is normal.

\begin{definition}\label{def:lc} Let $X$ be a $\Q$-Gorenstein normal variety, and let $f: Y\to X$ be a resolution of singularities. Write $$K_Y- f^*K_X\equiv   \sum a_i E_i,$$ where $E_i$ are the components of the exceptional divisor. We say $X$ is
\begin{enumerate}[(i)]
\item \emph{log canonical} if $a_i \geq -1$ for all $i$,
\item \emph{Kawamata log terminal} if if $a_i > -1$ for all $i$. \end{enumerate}\end{definition}

One can prove analogous results to those of this section in the semi-log canonical setting  (see Definition \ref{def:slc}), where $X$ is no longer irreducible, as follows. Let $X^{\nu}\to X$ be the normalisation, and let $\bar F$ be the conductor divisor (see Definition \ref{def:conductor}). Then uniform K-stability of $p: (X,L)\to (Y,H)$ is implied by uniform K-stability of $p^{\nu}: ((X^{\nu},\bar F);L) \to (Y,H)$ (this follows, for example, by \cite[Remark 3.19]{BHJ}). Thus the above result for pairs implies the analogue of the above result for semi-log canonical varieties.

\subsection{Odaka's blowing up formalism}\label{sec:odaka}

Remark that the Donaldson-Futaki invariant and the minimum norm are unchanged by modifying the line bundle $\L$ on $\X$ by adding a multiple of $\pi^*\scO_{\pr^1}(1)$. When dealing with these invariants in practice, it is often convenient to remove this ambiguity by choosing a more canonical choice of $\L$. A useful way of doing this is by using Odaka's blowing-up formalism \cite{odaka-RT}; this has the added benefit of giving a concrete geometric interpretation of the test-configurations. 

Note that one obtains resolutions of indeterminacy, as in equation \eqref{resolution}, by blowing up so-called flag ideals on $X\times\pr^1$, which are simply ideal sheaves of the form $$\scI = I_0 + I_1(t) + \hdots + (t^N),$$ where $t$ is the co-ordinate on $\pr^1$. Set $\B = Bl_{\scI}(X\times \pr^1)$ to be this blow-up with exceptional divisor $E$.  

\begin{proposition} A map $p:(X,L)\to (Y,T)$ is uniformly K-stable if and only if it is uniformly K-stable with respect to semi-test-configurations of the form $(\B, rL-E)$, where $r\geq 1$ is such that $rL-E$ is relatively semi-ample and $\B$ is normal. \end{proposition}

\begin{proof} This is essentially proved in \cite[Corollary 3.3]{RD-twisted} in the setting of ``twisted K-stability'' using the strategy of Odaka \cite{odaka-RT}. There it is proved that if there is a test-configuration satisfying $$\DF_p(\X,\L) < \epsilon \|(\X,\L)\|_m,$$ then there is a $(\B, rL-E)$ as above, \emph{not a priori admitting a map to} $Y$, with $$\DF_p(\B, rL-E) < \epsilon \|(\B, rL-E)\|_m$$ as a formal intersection number. However since there is a natural map $\B\to X\times\pr^1$, by composition it is clear that there is a map $\B\to Y$ extending the usual map on the general fibre.
\end{proof}

\begin{remark}\label{rmk:odaka-min}Note that, from Lemma \ref{lem:minimum} the Donaldson-Futaki invariant of such a semi-test-configuration is given up to a dimensional constant as $$\|(\B,rL-E)\|_m = (rL-E)^n.(L+nr^{-1}E).$$ \end{remark}

The main advantage of this formalism is the following concrete bounds on the intersection numbers defining the Donaldson-Futaki invariant. 

\begin{lemma}\cite[Proposition 4.3, Theorem 2.6]{Od-Sa} \cite[Equation (3)]{odaka-calabi} \cite[Lemma 3.7]{Derv}\label{inequalities} Suppose $\scI\neq (t^m)$ for any $m$, and let $R$ be a nef line bundle on $X$. With all notation as above, the following intersection theoretic inequalities hold: \begin{itemize} \item[(i)] $(rL-E)^n.R \leq 0$, \item[(ii)] $(rL-E)^n.E > 0$,  \item[(iii)] $(rL-E)^n.(rL+nE)>0$. \end{itemize}\end{lemma}

Remark in particular that $(iii)$ implies the minimum norm is strictly positive for non-trivial semi-test-configurations. We will need an improved version of this (which is proved in \cite[Lemma 4.40]{Dervan-Keller} when $n=2$). 

\begin{proposition}\label{odaka-improvement} With all notation as above, we have $$(rL-E)^n.(rL+(n-1)E)\geq 0.$$ \end{proposition}

\begin{proof}

Note that $L^{n+1}=0$ since $X$ is $n$-dimensional, while $L^n.E=0$ since $\scI$ has codimension two support in $X\times\pr^1$. A combinatorial formula then gives $$(rL-E)^n.(rL+(n-1)E) = -E.E.\left(\sum_{j=1}^{n-1}(n-j)(rL)^{j-1}.(rL-E)^{n-j}\right).$$ Now, $-E$ is relatively ample for the blowup map. Thus $-E.E$ is an effective cycle with support contained in the central fibre of $\B$, and hence each term of the sum is non-negative by (relative) semi-ampleness of $rL$ and $rL-E$. \end{proof}

\begin{remark}\label{rmk:min-improve}Although the improvement on Lemma \ref{inequalities} $(iii)$ seems minor, it is crucial to proving the results in the present section. Essentially the above says that, instead of the minimum norm just being positive, it is positive in an ``effective'' way: $$\|(\B,rL-E)\|_m \geq c^{-1}(rL-E)^n.E,$$ where $c$ is independent of the test-configuration.\end{remark}

\subsection{Kodaira embeddings}\label{subsec:stability}

\begin{theorem}\label{thm:embeddings-body}\ 
Assume that $X$ is log canonical.  Then any Kodaira embedding 
$$(X,L)\lhook\joinrel\xrightarrow{L^{\otimes k}}(\pr^{N_k},\scO_{\pr}(1))$$
is a uniformly K-stable map for $k\gg 0$. \end{theorem}

\begin{proof}  We use the blowing up formalism of Section \ref{sec:odaka}. Let $(\B,rL-E)$ be a semi-test-configuration as above. Since $X$ is log canonical, by inversion of adjunction and Lemma \ref{inequalities} $(ii)$ we have $$(rL-E)^n.K_{\B/X\times \pr^1}\geq 0,$$ hence by Remark \ref{rmk:odaka-min} it suffices to show that \begin{align*} \DF_p(\B,rL-E) &\geq  \frac{n}{n+1}\mu_{kL}(X,rL)(rL-E)^{n+1} + (rL-E)^n.(K_X+kL) \\ &\geq \epsilon(rL-E)^n.(L+nr^{-1}E) =  \epsilon \|(\B,rL-E)\|_m\end{align*} for some $\epsilon>0$ and $k\gg 0$. For notational convenience, denote $\mu:= \mu(X,L)$, so that $\mu_{kL}(X,rL) = r^{-1}\mu-r^{-1}k$. Setting $\hat{k} = \frac{k}{2(n+1)}$, the Donaldson-Futaki invariant can be rewritten as \begin{align*}& (rL-E)^n.\left(\frac{nr^{-1}}{n+1}(\mu-k)(rL-E)  +K_X+kL \right) , \\ & =(rL-E)^n.\left(\frac{nr^{-1}\mu}{n+1}(rL-E)+K_X+2\hat{k}r^{-1}(rL+nE) \right), \\ &\geq (rL-E)^n.\left(\frac{nr^{-1}\mu}{n+1}(rL-E)+K_X+2\hat{k}r^{-1}E \right),\\  &= (rL-E)^n.\left(\left(\frac{n\mu}{n+1}L+K_X+\hat kr^{-1} E\right)+r^{-1}\left (\hat{k}-\frac{n}{n+1}\right)E\right), \end{align*} where we have used Proposition \ref{odaka-improvement} to go from the second to the third line. 

We deal with the two terms separately. For the first term, note that $$(rL-E)^n.E \geq (rL-E)^n.(-n^{-1}rL)$$ by Lemma \ref{inequalities} $(iii)$, so that \begin{align*} (rL-E)^n.\bigg(\frac{n\mu}{n+1}&L+K_X+\hat k r^{-1}E\bigg) \\ & \geq (rL-E)^n.\bigg(\frac{n\mu}{n+1}L+K_X-\hat k n^{-1}L\bigg),\end{align*} which is positive proved $-\frac{n\mu}{n+1}L-K_X+\hat k n^{-1}L$ is nef by Lemma \ref{inequalities} $(i)$. Certainly the class is nef for $\hat k$ (or equivalently $k$) sufficiently large. 

For the second term, note that by Lemma \ref{inequalities} $(ii)$ and $(iii)$, for $\hat k \geq 1$ we have 
\begin{align}(rL-E)^n.\left(r^{-1}\left (\hat{k}-\frac{n}{n+1}\right)E\right)&\geq \frac{1}{n+1}(rL-E)^n.r^{-1}E, \\ &\geq \frac{1}{n(n+1)}(rL-E)^n.(L+nr^{-1}E), \end{align} as required.  \end{proof}

\begin{remark} The same proof shows that the identity map $(X,L) \to (X,L^k)$ is a uniformly K-stable map for $k \gg 0$. This can also be obtained from a combination of the above and the factorisation properties of K-stability of maps, which we shall prove as Theorem \ref{thm:factorization} $(i)$. \end{remark}

It is natural to ask how $k$ depends on $(X,L)$, and whether or not the above construction can be performed in families. Our next result shows that this is the case, at least when $X$ is smooth, if one instead embeds using powers of certain adjoint bundles $mL+2K_X$. 

\begin{theorem}\label{thm:uniform-kodaira}

Consider the set $\mathcal{S}$ of smooth polarised varieties with values $\{\dim X = n, \vol(X) = L^n, L^{n-1}.K_X, L^{n-2}.K_X^2,\hdots, K_X^n\}$ fixed. For all $m \gg 0$, there is a $k=k(m)$ such  that for all $(X,L)\in \mathcal{S}$ the map $$(X,mL+2K_X)\lhook\joinrel\xrightarrow{(mL+2K_X)^{\otimes k'}}(\pr^{N_{k'}},\scO_{\pr}(1))$$ is uniformly K-stable for all $k' \geq k$.

\end{theorem}

\begin{proof}

This is essentially a consequence of the various results towards Fujita's conjecture. We use Demailly's result, which states that there exists an $m'$ such that for all $(X,L) \in \mathcal{S}$, the line bundle $mL+2K_X$ is very ample for all $m \geq m'$ \cite[Corollary 2]{demailly-numerical}. We shall work only with such $m$. 

The condition in the proof of Theorem \ref{thm:embeddings-body} that needs to be satisfied is that the line bundle  $$-\frac{n}{n+1}\mu(X,mL+2K_X)(mL+2K_X) -K_X + \hat k n^{-1}(mL+2K_X)$$ is nef for $m \gg 0$ and $\hat k \gg 0$, where explicitly $$\mu(X,mL+2K_X) = \frac{-K_X.(mL+2K_X)^{n-1}}{(mL+2K_X)^n}.$$ Grouping terms, this line bundle is $$\left(-\frac{n}{n+1}\mu(X,mL+2K_X) +  \hat k n^{-1}-1\right)(mL+2K_X)+mL.$$ 

Note that for $m\gg 0$, independent of $(X,L)\in \mathcal{S}$, we have $\mu(X,mL+2K_X)\leq 1.$ Hence for any such $m$, the number $-\frac{n}{n+1}\mu(X,mL+2K_X) +  \hat k n^{-1}-1$ is positive for all $\hat k \gg 0$ (or equivalently $k \gg 0)$. Thus the result follows since $mL+2K_X$ is very ample, hence nef. 

\end{proof}

This is a sort of boundedness result for K-stable maps. A version of the Fujita conjecture for log canonical varieties would extend the above proof to the singular setting. Of course, in the singular case for any geometric application of the above it is crucial to obtain bounds on the Cartier index of $L$. For example, a typical case is when $L=\pm K_X$ (so $\mathcal{S}$ just fixes the dimension and the volume), where to obtain some boundedness result the main difficulty is to bound the Cartier index of $\pm K_X$. The point of the above Corollary is that, once one has a bounded family of varieties, it is essentially automatic that one obtains a bounded family of K-stable maps. 

One can easily generalise Theorem \ref{thm:embeddings-body} as follows, using the notion of ``J-stability'' introduced by Lejmi-Sz\'ekelyhidi \cite{Lejmi-Sz}. We use the reformulation of \cite[Proposition 4.29]{Dervan-Keller}. Denote \begin{equation}\label{eq:gamma}\gamma_T(X,L) = \frac{L^{n-1}.T}{L^n}.\end{equation}

\begin{definition}\label{def:j-stab} Consider a variety $X$ with line bundles $L,T$ where $T$ is ample. We say that $((X,T);L)$ is \emph{uniformly J-stable} if there exists an $\epsilon>0$ such that for each test-configuration $(\X,\L)$, we have $$\J_T(\X,\L) = -\frac{n}{n+1}\gamma_T(X,L) \L^{n+1} + \L^n.T \geq \epsilon \|(\X,\L)\|_m,$$ calculated on a resolution of indeterminacy. \end{definition}

\begin{remark} When $X$ is smooth and $T$ is ample, uniform J-stability is conjecturally equivalent to the existence of a solution to $\Lambda_{\omega} \alpha = c$, where $c\in \R$, $\alpha\in c_1(T)$ is a fixed K\"ahler form and $\omega\in c_1(L)$ \cite[Conjecture 1]{Lejmi-Sz}.\end{remark}

\begin{theorem}\label{thm:j-stab} Suppose $X$ is log canonical, $((X,T);L)$ is uniformly J-stable and $T$ is semi-ample. Then $$(X,L)\joinrel\xrightarrow{T^{\otimes k}}(\pr^{N_k'},\scO_{\pr}(1))$$ is a uniformly K-stable map for all $k\gg 0$. Conversely, if $((X,T);L)$ is J-unstable, then the Kodaira embedding above is a K-unstable map. \end{theorem}

\begin{proof} The proof is essentially identical to that of Theorem \ref{thm:embeddings-body}, by noting that for a blow-up test-configuration we have $$\DF_p(\B,rL-E) =  \J_{K_X+kT}(\B,rL-E)+ (rL-E)^n.K_{\B/X\times \pr^1},$$ and using the linearity property \begin{equation}\label{eq:j-linearity}\J_{aH+bM}(\B,rL-E) =a\J_{H}(\B,rL-E)+ a\J_{M}(\B,rL-E).\end{equation} \end{proof} It is fairly simple to give explicit criteria for uniform J-stability \cite{Lejmi-Sz,Dervan-Keller,Yoshi-Keller}, so from the above we obtain further examples of stable maps. 

\subsection{Numerical Properties, and Variation}

Here we prove some results regarding the behaviour of uniform K-stability on the various cones of line bundles of $X$ and $Y$. These results rely on the \emph{uniform} lower bound on the Donaldson-Futaki invariant, it seems much more challenging to prove them assuming only that the map is K-stable. Firstly, we show that uniform K-stability is a numerical condition on $L$ and $T$:

\begin{theorem} \label{thm:numerical} Uniform K-stability of $p:(X,L)\to (Y,T)$ depends only on the numerical class of $L,T$. \end{theorem}

\begin{proof} This is obvious for $T$. For $L$, this a consequence of the blowing up formalism of  Section \ref{sec:odaka}. Suppose $L'\equiv_{num} L$. If $p:(X,L)\to (Y,T)$, then for all $\epsilon>0$ there is a semi-test-configuration $(\B,rL-E)$ with $\DF_p(\B,rL-E)< \epsilon \|(\B,rL-E)\|_m$. By definition of a semi-test-configuration, $rL-E$ is relatively semiample. Perturbing slightly one can assume $rL-E$ is actually relatively ample, while preserving the inequality $\DF(\B,rL-E)< \epsilon \|(\B,rL-E)\|_m$. But then as relative ampleness is a numerical condition,  $(\B,rL'-E)$ is a test-configuration for $p': (X,L')\to (Y,T)$ and one still has $\DF_p(\B,rL'-E) < \epsilon \|(\B,rL'-E)\|_m$, as required.  \end{proof}

\begin{remark} Given the above, it is natural to ask if there is a definition of K-stability of maps when $L\in \Amp_{\R}(X)$ and $T\in\Pic_{\R}(Y)$. In fact more generally one can give a definition that makes sense for $L$ replaced a \emph{K\"ahler} class with $X$ a smooth K\"ahler manifold and $Y$ a complex manifold, and $T$ replaced with a Bott-Chern class, by a straightforward variant of the K\"ahler version of K-stability defined in \cite{Dervan-Ross,ZSD}. \end{remark}

We next prove an openness result for uniformly K-stable maps. 

\begin{theorem}\label{thm:variation} Let $p:(X,L)\to (Y,H)$ be a uniformly K-stable map. Then stability of a map is an open condition in $\Pic_{\Q}Y$. \end{theorem}

\begin{proof}

We use the blowing up formalism of Section \ref{sec:odaka}, and follow the notation used there and in Definition \ref{def:j-stab}. Following this notation, note that $$\J_L(\B,rL-E) = \|(\B,rL-E)\|_m.$$ Fix some $T\in \Pic_{\Q}(Y)$. Since $ \Pic_{\Q}(Y)$ is a finite dimensional vector space, from the definition of uniform K-stability, it is clear that it suffices to establish that there exists a $c=c(T)\in \R$ independent of $(\B,rL-E)$ such that $$-c\J_L(\B,rL-E) \leq \J_T(\B,rL-E) \leq c \J_L(\B,rL-E).$$ 

We first establish the second inequality. By the linearity property noted in equation \eqref{eq:j-linearity}, it is enough to show that for $c\gg 0$ we have $$\J_{cL-T}(\B,rL-E)\geq 0.$$ In fact our argument will still hold if we replace $T$ with $-T$, so will give the first inequality as well. 

By definition and using $\gamma_{cL-T}(X,L) = c-\gamma_T(X,L)$  we have \begin{align*}\J_{cL-T}(\B,rL-E) &= (rL-E)^n.\left(-\frac{nr^{-1}}{n+1}\gamma_{cL-T}(X,L)(rL- E) + cL-T\right), \\ &= (rL-E)^n.\bigg(\frac{r^{-1}}{n+1}\gamma_{cL-H}(rL+(n-1) E)- \\ & \ \ \ \ \ \ \ \ \ \ \ \  \ \ \ \ \ \  \ \ \ \ \ \ -\gamma_{T}L-T +\frac{1}{n+1}(c-\gamma_{T})E  \bigg) .\end{align*} For $c \gg 0$ we have $\gamma_{cL-T}(X,L)>0$ so the first term is non-negative by the key result Proposition \ref{odaka-improvement}. The remaining terms sum to a non-negative number by a combination of Lemma \ref{inequalities} $(iii)$ and Lemma \ref{inequalities} $(i)$.

\end{proof}

\begin{remark} The above result also shows that uniform J-stability in the sense of Definition \ref{def:j-stab}  is an open condition as one varies $T$. Remark again that the improvement in Proposition \ref{odaka-improvement} compared to Lemma \ref{inequalities} $(iii)$ is crucial in proving  the above: knowing only Lemma \ref{inequalities} $(iii)$, the proof breaks down.  \end{remark}

 \subsection{Factorisation, compositions, naturality}\label{sec:fact}

Here we prove:
 
\begin{theorem}\label{thm:factorization} Let $(X,L)\joinrel\xrightarrow{p} (Z,q^*T)\joinrel\xrightarrow{q} (Y,T)$ be maps.
\begin{enumerate}[(i)]
\item  If $p\circ q: (X,L) \to (Y,T)$ is K-stable (resp. K-semistable, uniformly K-stable), then $X\to Z$ is K-stable (resp. K-semistable, uniformly K-stable).  If $q: Z\to Y$ is an isomorphism, then the converse is true. Thus the automorphism group of $Y$ acts on the space of stable maps to $Y$.
\item If $p: (\X,\L)\to (Z,q^*T)$ is uniformly K-stable, then so is $q\circ p: (X,L) \to (Y,T)$.
\end{enumerate}
\end{theorem}

\begin{proof} \

$(i)$ Note that if $p: (\X,\L)\to (Z,q^*T)$ is a test-configuration for $p: (X,L)\to (Z,q^*T)$, then $q\circ p: (\X,\L)\to (Y,T)$ is a test-configuration for $q\circ p: (X,L) \to (Y,T)$. The result follows easily from this. 

$(ii)$ Let $p: (\X,\L)\to (Y,T)$ be a test-configuration for $p\circ q: (\X,\L)\to (Y,T)$ which satisfies $$\DF_{q\circ p} (\X,\L) < \epsilon \|(\X,\L)\|_m.$$ Take an equivariant resolution of indeterminacy as in equation \eqref{resolution}: \begin{equation}\label{res-ind}
\begin{tikzcd}
\Y \arrow[swap]{d}{h} \arrow{dr}{f} &  \\
X\times\pr^1 \arrow[dotted]{r}{g} & \X
\end{tikzcd}
\end{equation}

The line bundle $f^*\L$ is relatively semi-ample over $\pr^1$. Letting $E$ be the exceptional divisor of $f$, the line bundle $\L-\delta E$ is relatively ample for $\delta$ sufficiently small. The natural map $h \circ p: (\Y,\L-\delta E)\to (Z,q^*T)$ gives a test-configuration for $q\circ p: (\X,\L)\to (Z,q^*T)$. Moreover, for $\delta$ sufficiently small, by continuity of the intersection numbers defining the Donaldson-Futaki invariant and the minimum norm, we have $\DF_{p}(\Y,\L-\delta E) < \epsilon \|(\Y,\L-\delta E)\|_m.$ Since this argument works for all $\epsilon$ sufficiently small, this contradicts the uniform K-stability of $p: (\X,\L)\to (Z,q^*T)$, proving the result.

\end{proof}

The same argument gives the following, as promised in Remark \ref{rmk:twistedKstability}:

\begin{corollary} A map $p: (X,L)\to (Y,T)$ is uniformly K-stable if and only if $(X,L)$ is uniformly twisted K-stable with respect to $p^*T$. \end{corollary}

\begin{proof} The definition of twisted K-stability of $(X,L)$ involves taking a resolution of indeterminacy of a test-configuration $\Y\to \X$ as in \eqref{res-ind} above, and defining $$\DF_T(\X,\L):= \frac{n}{n+1}\mu_T(X,L^r)\L^{n+1} + \L^n.(K_{\X/\pr^1} + f^*T).$$ By the argument of Theorem \ref{thm:factorization}, by perturbing we can assume $(\X,\L)$ itself admits a map to $(X,L)$, hence a map to $(Y,T)$ by composing. This invariant clearly equals the Donaldson-Futaki invariant of $p: (X,L)\to (Y,H)$, hence uniform K-stability of $p$ implies uniform twisted K-stability of $(X,L)$ with respect to $p^*T$. The converse is obvious. \end{proof}
  
 \section{Fibrations}\label{sec:fibrations}
 
Let $f:U\to Y$ be a flat proper morphism between schemes of constant relative dimension $d$ and $L_{U}$ be a line bundle on $U$ that is relatively ample over $Y$.  We recall the construction of the CM-line bundle (see \cite[Section 2]{FR-CM} for a more detailed account).  The Knudson-Mumford  expansion \cite[Theorem 4]{KM-expansion} provides line bundles $\lambda_i$ for $i=0,\ldots, d$ on $Y$ and a natural polynomial expansion
$$ \det(\pi_{!} L_U^k) \simeq \lambda_{d+1}^{\binom{k}{d+1}} \otimes \lambda_{d}^{\binom{k}{d}} \otimes \cdots \otimes \lambda_0 \text{ for } k\ge 0.$$
Moreover the $\lambda_i$ commute with base-change.  The \emph{CM-line bundle} \cite[Definition 1]{PT-CMstability} with respect to $L_U$ is defined to be
\begin{equation}L_{CM} = \lambda_{d+1}^{d \mu + d(d+1)} \otimes \lambda_{d}^{-2(d+1)}\label{eq:defCMlinebundle}\end{equation}
where $\mu=\mu(U_y,L_U|_{U_y})$ is the slope of any fibre of $U$ (and the reader is warned our convention for $\mu$ differs to that of \cite{FR-CM}).

 Now suppose $p:B\to Y$ is a morphism from a normal projective variety $B$ and consider the fibre product

$$\begin{tikzcd}
X: = B\times_Y U \arrow{r}{p} \arrow{d}{f}   & U \arrow{d}{f}\\
B \arrow{r}{p} &Y
\end{tikzcd}$$

 \begin{remark}
The reader should have in mind here the case that $Y$ is some kind of moduli space of varieties or schemes with a universal family $U$.  Then $p:B\to Y$ carries the same data as the fibre product $X\to B$, which is a fibration whose fibres vary in the moduli space $Y$.
 \end{remark}

 Fixing an ample line bundle $L_B$ on $B$,   for $m$ sufficiently large the line bundle
 $$ L_X:= p^*L_U \otimes f^*L_B^m$$
 on $X$ is ample.  We wish to relate K-stability of $(X,L_X)$ with K-stability of the morphism 
$$p:(B,H)\to (Y,\delta L_{CM})$$  for some suitable constant $\delta=\delta(m)>0$ and $m\gg 0$.  To this end, suppose that $(\mathcal B,\mathcal L_\mathcal B,p)$ is a test-configuration for $p$ and set
$$\begin{tikzcd}
\mathcal X: = \mathcal B\times_Y U \arrow{r}{p} \arrow{d}{f}   & U \arrow{d}{f} \\
\mathcal B \arrow{r}{p} &Y
\end{tikzcd}$$
and
$$L_{\mathcal X}:= p^*L_U \otimes  f^*\mathcal L_{\mathcal B}^m.$$

\begin{lemma}
For $m$ sufficiently large $(\mathcal X,L_{\mathcal X})$ is a test-configuration for $(X,L)$.
\end{lemma}
\begin{proof}
As flatness commutes with basechange, the fact that $U$ is flat over $Y$ and $\mathcal B$ is flat over $\mathbb P^1$ imply that $\mathcal X$ is also flat over $\mathbb P^1$.  The $\mathbb C^*$-action on $\mathcal B$ lifts to $(\mathcal X, L_\mathcal X)$, and the properties needed to make this data a test-configuration are all easily verified.
\end{proof}
Now let $V = L_{X_b}^n$ where $X_b$ is a fibre of $X\to B$ over some (resp.\ any) point $b\in B$ (so $V$ is the volume of the fibre)  and set
$$ \delta = \frac{1}{(n-b+1)V}.$$

\begin{proposition}\label{prop:DFfibration}
It holds that
\begin{equation}DF(\mathcal X,\mathcal L_{\mathcal X})  = V \binom{n}{b}  DF_{\delta p^* L_{CM}}(\mathcal B,\mathcal L_{\mathcal B}) m^{b} + O(m^{b-1}).\label{eq:DFfibration}\end{equation}
\end{proposition}

\begin{corollary}\label{cor:Kstabfibration}
For $m$ sufficiently large, if $(X,L_X)$ is a K-semistable variety then $p:(B,L_B)\to (Y,\delta L_{CM})$  is a K-semistable map.
\end{corollary}

\begin{proof}[Proof of Proposition \ref{prop:DFfibration}]
Let $\dim B=b$ and $\dim X = n$ and set
$$ \mu: = \mu(X_b,L_{X}|_{X_b}) = -\frac{K_{X_b}.L_{X_b}^{n-b-1}}{L_{X_b}^{n-b}}$$
where $X_b$ denotes a (resp.\ any) fibre of $X\to B$, which has dimension $n-b$.    The CM line bundle commutes with base change, so $p^*L_{CM}$ is precisely the CM line bundle of the fibration $f:X\to B$ computed with $p^*L_U$.  The first Chern class of the CM line bundle is easily calculated with Grothendieck-Riemann-Roch \cite[Corollary 18.3.1 (c)]{fulton-book}, which gives \cite[p2]{FR-CM}
\begin{equation} c_1(p^*L_{CM}) = f_* [(n-b)\mu c_1(p^*L_U)^{n-b+1} + (n-b+1) c_1(K_{X/B}) c_1(p^*L_U)^{n-b}].\label{eq:chernclasscm}\end{equation}
Here $c_1(K_{X/B})$ is the cycle which is the degree two part of the relative singular Todd class \cite[p354]{fulton-book}.  Our first task is to calculate $\mu(X,L_X)$ asymptotically for large $m$.  To ease exposition we shall drop the pullback and use additive notation, so $L_X = L_U + m L_B$.  Then $L_X^n = (L_U + mL_B)^m = \binom{n}{b} m^b L_B^b L_U^{n-b} + \binom{n}{b-1} m^{b-1} L_B^{b-1} L_U^{n+1-b} + O(m^{b-2})$ and similarly for $L_X^{n-1}$.    Algebraic manipulation and the  projection formula yields
\begin{equation}\label{eq:expansionslopefibration}
 \mu(X,L_X) = \frac{-K_X. L_X^{n-1}}{L_X^n} = \lambda_0 + \lambda_1 m^{-1} + O(m^{-2})
 \end{equation} where
\begin{align*}
\lambda_0 &= \frac{n-b}{n} \mu\\
\lambda_1&= -\frac{b}{n(L_B^b)}( \delta L_B^{b-1} p^* L_{CM} + L_B^{b-1}K_B) = \frac{b}{n} \mu_{T}(B,L_B)\end{align*}
where $T=\delta p^*L_{CM}$.  A similar calculation yields the desired Donaldson-Futaki invariant, which by definition is equal to
\begin{align*}
DF(\mathcal X,\mathcal L_\mathcal X) &= \frac{n}{n+1} \mu(X,L) \mathcal L_\mathcal X^{n+1} + K_{\mathcal X/\mathbb P^1} \mathcal L_{\mathcal X}^n\\
&=\frac{n}{n+1}(\lambda_0 + \lambda_1 m^{-1} + O(m^{-2}))\mathcal L_{\mathcal X}^{n+1} + K_{\mathcal X/\mathbb P^1} \mathcal L_{\mathcal X}^n.\end{align*}
Expand $\mathcal L_{\mathcal X}^{n+1} = (p^* L_U + mf^* \mathcal L_{\mathcal B})^{n+1}$ and extract the top two powers of $m$, and similarly for $\mathcal L_{\mathcal X}^{n}$.  Algebraic manipulation yields that the $m^{b+1}$ term in $DF(\mathcal X,\mathcal L_\mathcal X)$ vanishes, and the $O(m^{b})$ term is  
$$\binom{n}{b}\left( \mathcal L_\mathcal B^b K_{X/\mathbb P^1} p^*L_U^{n-b} + \frac{b}{b+1} \mu_T(B,L_B) \mathcal L_B^{b+1} p^*L_U^{n-b} + \frac{\mu(n-b)}{n+1-b} \mathcal L_\mathcal B^n p^* L_U^{n+1-b}\right).$$
Along with the observation that $\mathcal K_{\mathcal X/\mathbb P^1} = K_{\mathcal X/\mathcal B} + K_{\mathcal B/\mathbb P^1}$, an application of the projection formula along with the formula for $c_1(p^*L_{CM})$ \eqref{eq:chernclasscm} yields \eqref{eq:DFfibration}.
\end{proof}

\begin{proof}[Alternative Proof of Proposition \ref{prop:DFfibration}]
We sketch a proof that does not require Grothen-dieck-Riemann-Roch.  Again $\dim X=n$ and $\dim B=b$.  First observe that if $E$ is a vector bundle of rank $r_E$ on $(B,L_B)$ then the Euler-characteristic satisfies
$$ \chi(E\otimes L_B^p) = r_E \chi(L^p) + \frac{p^{b-1}}{(b-1)!} \int_B c_1(E) c_1(L)^{b-1}+ O(p^{b-1}).$$
(This can be seen by assuming $L_B$ to be very ample and taking hyperplane sections, but of course can also be seen from Grothendieck-Riemann-Roch).  Thus if $E_k$ is a sequence of vector bundles of rank $r_{k}$ then
$$ \chi(E_k\otimes L^{\otimes mk}) = r_k \chi(L_B^{mk}) + \frac{m^{b-1} k^{b-1}}{(b-1)!} \int_B c_1(E_k) c_1(L^{b-1})+O(m^{b-2})$$
where the $O(m^{b-2})$ term also depends on $k$.  Now let $E_k = \pi_{!}(L_U^{\otimes k})$ for large $k$.  The using the notation from \eqref{eq:defCMlinebundle} there are line bundles $\lambda_i$ on $Y$ such that
$$\det(E_k) = \lambda_{n-b+1}^{\binom{k}{n-b+1}} \otimes \lambda_{n-b}^{\binom{k}{n-b}}\otimes \cdots \otimes \lambda_{0}.$$
Hence by the projection formula
\begin{align*}\chi(L_X^k) &= \chi(p^* L_U^k \otimes\pi^* L_B^{mk}) = \chi(\pi_{!} p^*L_U^k \otimes L_B^{mk}) \\
&=r_k \chi(L_B^{mk}) + \frac{m^{b-1} k^{b-1}}{(b-1)!}  \int_B c_1(p^*E_k) c_1(L_B^{b-1}) +O(m^{b-2})\\
&=r_k \chi(L_B^{mk}) + \frac{m^{b-1} k^{b-1}}{(b-1)!} \left( \binom{k}{n-b+1} \int_B c_1(p^*\lambda_{n-b+1}) c_1(L_B)^{b-1}\right)\\
&+\frac{m^{b-1} k^{b-1}}{(b-1)!} \left(\binom{k}{n-b} \int_B c_1(p^*\lambda_{n-b}) c_1(L_B)^{b-1}\right) +O(k^{n-2}) + O(m^{b-2}).
\end{align*}
Now $r_k = rank (E_k)$ is the Hilbert-polynomial of the fibre of $U\to Y$, and so one can extract the $m^b$ term and $m^{b-1}$ term in the top two leading order terms of $\chi(L_X^k)$ in $k$.  Algebraic manipulation then gives the expansion of the slope $\mu(X,L)$ is as stated as in \eqref{eq:expansionslopefibration}.  The proof for the Donaldson-Futaki invariant is a similar calculation on the total space $\mathcal X$, and is left to the reader.\end{proof}

\begin{remark}
In the above we do not assume that $L_{CM}$ has any positivity, and in fact there are examples for which it is negative \cite[Example 5.2]{FR-CM}.  This is the only case we know where K-stability of a map $p:(X,L)\to (Y,T)$ may be interesting without any positivity assumptions on $T$.
\end{remark}

\begin{remark}[Converse] The converse to Corollary \ref{cor:Kstabfibration}  clearly requires some stability hypothesis of the fibres of $X$ (as can be seen if $X$ is a product).   We speculate that with some such hypothesis (for instance if one assumes they are canonically polarised or uniformly K-stable) then stability of $X\to C$ is equivalent to stability of the map $p: C\to Y$ (either assuming canonical polarisations, or otherwise taking $m$ to be sufficiently large).

The two difficulties in proving such a statement are (i) the fact that a priori there can be test-configurations for $X$ that have limits that are not themselves fibrations and (ii) how large $m$ must be taken should be uniform over all test-configurations for $X$ that need to be considered.   This may be related to the fact that if a KSBA stable variety admits a fibration to a stable base with stable fibres then this fibration structure deforms uniquely for small deformations \cite{Patakfalvi-fibered}.

\end{remark}

\begin{remark}[Stacks]\label{rmk:stacks}As the reader is surely aware, in general moduli spaces do not come with universal families due to the presence of automorphisms.   But one can run the same argument as above (which is purely formal) if $Y$ is instead taken to be a Deligne-Mumford stack.   The main difference is that a test-configuration for a morphism $B\to Y$ will itself be a stack, but one can define the Donaldson-Futaki invariant in precisely the same way as before (for instance using the same intersection formula \eqref{eq:DFintersection}). 

For example, if $B$ is a curve and $Y=\M_g$ the (proper) moduli stack of stable curves (of some fixed genus say) then $X\to B$ is a fibered surface whose stability is related to stability of the map $p:B\to Y$.  For another example, $Y$ could be the KSBA moduli stack of canonically polarised semi-log-canonical varieties. By definition of a morphism of stacks, from a map $\B\to Y$ one obtains a family $\X\to \B$ whose fibres are KSBA stable varieties. Thus if $B\to Y$ is a K-unstable map (of stacks), $(X,L_X)$ is K-unstable for $m \gg 0$, without any further hypotheses needed. This suggests that for the study of stability of fibrations, the more useful notion of K-stability of maps should allow maps to stacks.\end{remark} 

\begin{remark}[Projective bundles] \label{rmk:bundles}Another examples of a fibration that has attracted significant interest from the point of view of K-stability and canonical K\"ahler metrics is that of the projectivisation $\mathbb P(E)$ of a vector bundle $E$ over $(B,L_B)$ (for instance \cite{apostolov-remark,apostolov-ruled,aposolovIII,bronnle-extremal,Hong-hermitian,keller-mumfordsemistable,KR-notechow,li-extremal,RT-obstruction}).

From the point of view of this paper it makes sense to consider the moduli stack $\M$ of projective space (of course the coarse moduli space of $\M$ is a single point, but the stack which is clearly not Deligne-Mumford is much richer).  Then any projective bundle $\mathbb P(E)\to B$ is induced by a map $p:B\to \M$ which must be K-semistable if $(\mathbb P(E), mL_B + \scO_{\pr(E)}(1))$ is K-semistable for $m\gg 0$.  This is slightly different, but presumably related to, requiring that $(B,L_B)$ be stable and that $E$ be a stable vector bundle, which are the kind of hypothesis usually made in the references above.

\end{remark}

\begin{remark}[Comparison with Abramovich-Vistoli]
The compactification of the moduli space of stable fibred surfaces $X\to B$ is considered by Abramovich-Vistoli in \cite{ AbramovichVistoli-complete,AbramovichVistoli-compactifying}.  In that paper the authors compactify the space of fibrations $X\to B$ by stable curves over a one dimensional base $B$ such that the induced map $B\to M_g$ is a stable map (in the sense of Kontsevich).   The points in the boundary of their moduli space consist of certain maps $\tilde{B}\to \mathcal M_g$ where $\mathcal M_g$ is the moduli stack of curves, and $\tilde{B}$ is a curve endowed with additional stack structure.   This further suggests that it is interesting to consider K-stability of maps whose \emph{domain} is a stack.  We refer to \cite{RT-orbifold} for prior work towards K-stability for certain Deligne-Mumford stacks in the absolute case. The generalisation to pairs is taken up in \cite{Ascher}.
\end{remark}

\section{Moduli spaces of maps}

\subsection{Preliminaries on semi-log canonical pairs}

We recall some definitions and results we which require from the minimal model program. Most importantly, we shall define semi-log canonical (or slc) varieties, which are the higher dimensional analogue of nodal curves.

As such varieties are typically not irreducible, the most effective way to study them is through their normalisation.  Recall that a nodal curve $C$ is encoded by the triple $(\bar C, \bar F, \tau)$, where $\bar C$ is its normalisation, $\bar F$ is the preimage of the nodes and $\tau: \bar F \to \bar F$ is the involution which determines which pairs of points are identified in $C$. We will use a similar technique to study slc varieties, following closely ideas of Koll\'ar \cite[Section 5]{kollar-singularities}.

\begin{definition}\label{def:conductor}
Let $(Y,D)$ be a pair consisting of a normal variety $Y$ and an effective $\Q$-Weil divisor $D$ such that $K_X+D$ is $\Q$-Cartier. Let $f: Y\to X$ be a log resolution of singularities, so that $f_*^{-1}D\cup E$ has simple normal crossing singularities. Write $$K_Y- f^*(K_X +D) = \sum a_i E_i,$$ where either $E_i$ is exceptional or the proper transform of a component of $D$. We say $(X,D)$ is \emph{log canonical} if $a_i \geq -1$ for all $i$.

Now let $X$ be an equidimensional $\Q$-Gorenstein projective variety. We say $X$ is \emph{demi-normal} if it satisfies Serre's S2 condition and is nodal in codimension one. For a demi-normal variety $X$, denote by $\pi: \bar{X} \to X$ its normalisation. The preimage of the double normal crossing locus of $X$ is called the \emph{conductor} of $\pi$ and denoted $\bar F$. We say that $X$ is \emph{semi-log canonical} if $(\bar X, \bar F)$ is log canonical. \end{definition}

The map $\bar F \to F$ induces an involution between the normalisations $\tau: \bar F^{\nu} \to \bar F^{\nu}$, which is fixed point free in codimension one \cite[p189]{kollar-singularities}. For this it is essential to work on the normalisation of $\bar F$: there are examples in which a point in $F$ has three preimages in $\bar F$, so no involution can exist \cite[p189]{kollar-singularities}. Then as in the curve case, the triple $(\bar X, \bar F, \tau)$ determines $X$ \cite[Theorem 5.13]{kollar-singularities}. Moreover, we have \cite[Equation (5.7.5)]{kollar-singularities} \begin{equation}\label{divisorless-normalisation}\pi^*(K_X+F) \sim_{\Q} K_{\bar X} + \bar F.\end{equation}

We will require a similar technique for pairs, so let $D$ be a $\Q$-Weil divisor on $X$ whose support does not contain any codimension one component of the singular locus of $X$. Then $D$ is $\Q$-Cartier in codimensione one, so we can define $\bar D$ as the closure of the pullback $\pi^*D$ on the $\Q$-Cartier locus.

\begin{definition}\label{def:slc} We say a pair $(X,D)$ is \emph{semi-log canonical} if $K_X+D$ is $\Q$-Cartier and $(\bar X, \bar D+ \bar F)$ is log canonical. \end{definition} The analogue of equation (\ref{divisorless-normalisation}) for pairs states that $$\pi^*(K_X+D) \sim_{\Q} K_{\bar X} + \bar F + \bar D.$$

In order to recover $(X,D)$ from its normalisation, we need to consider the corresponding involution. Choose $m$ such that $mD$ is integral and $m(K_{\bar X}+\bar F+\bar D)$ is Cartier. For $\sigma: \bar F^{\nu}\to \bar X$ the induced map, one defines an effective $\Q$-Cartier divisor on the normalisation $\bar F^{\nu}$ called the \emph{different}, denoted $\Diff_{\bar F^{\nu}}(\bar D)$, in such a way that:
\begin{enumerate}[(i)]
\item $m\Diff_{\bar F^{\nu}}(D)$ is integral and $m(K_{\bar F^{\nu}}+\Diff_{\bar F^{\nu}}(\bar D))$ is Cartier,
\item $\sigma^*\pi^*(K_X+D) \sim_{\Q}\sigma^*(K_{\bar X} + \bar F + \bar D)\sim_{\Q} K_{\bar{F}^{\nu}}+\Diff_{F^{\nu}}(\bar D).$
\end{enumerate} We refer to \cite[Section 5.11]{kollar-singularities} for further details. Again, the key point is that $\Diff_{\bar F^{\nu}}(\bar D)$ is a $\tau$-invariant divisor, and the data $(\bar{X},\bar F+\bar D)$  together with the involution $\tau$ of $(\bar F^{\nu}, \Diff_{\bar F^{\nu}}(\bar D))$ determines $(X, D)$ \cite[Theorem 5.38]{kollar-singularities}. A simple consequence of the definition is the following. 

\begin{lemma} Suppose $K_X+D$ is ample. Then so is $K_{\bar{F}^{\nu}}+\Diff_{F^{\nu}}(\bar D)$. \end{lemma}

\begin{proof} This follows immediately from the formula defining the different, noting that $\pi\circ\sigma$ is finite.\end{proof}

We will also need some information on the singularities of the pair $(\bar F^{\nu}, \Diff_{\bar F^{\nu}}(\bar D))$.

\begin{proposition} Suppose $(X,D)$ is slc. Then the pair $(\bar F^{\nu}, \Diff_{\bar F^{\nu}}(\bar D))$ is log canonical. \end{proposition}

\begin{proof} This follows from adjunction, since $(\bar X, \bar D+ \bar F)$ is log canonical. \end{proof}

\subsection{Moduli}

Let X be a variety and let $D$ be a divisor such that the pair $(X,D)$ is slc. In this section we construct a moduli space of stable maps $p: ((X,D);L)\to (Y,H)$ such that $L=K_X +D+p^*H$ is an ample $\Q$-Cartier divisor. We sometimes abbreviate this data to $(X,D)$ and simply call this a  ``stable map''. Remark \ref{rmk:generaltype} ensures that these are simply canonically polarised uniformly K-stable maps. It will be important in our construction that $D$ is a genuine Weil divisor, rather than merely a $\Q$-Weil divisor. 

\begin{remark}\label{rmk:suffample}For technical reasons, we require that $H$ is ``sufficiently ample''. Precisely, fixing an arbitrary ample line bundle $M$ on $Y$, we will require that $H$ satisfies $H- 2n M$ is ample, where $n=\dim X$. For example, this applies for $H=(2n+1)M$.\end{remark}

Our main result is as follows. 

\begin{theorem}\label{moduli-space-existence} The moduli functor of stable maps is coarsely represented by a separated projective scheme.\end{theorem}  

The construction of this moduli space is due to Kontsevich in the case $n=1$ and Alexeev in the case $n=2$ \cite{va-icm,va-stablemaps}. Our proof uses several deep results from the minimal model program, and follows Alexeev's strategy in the case $n=2$. The main differences compared with Alexeev's work arise due to the progress made in the minimal model program over the last twenty years. In particular the moduli space we construct has a slightly different scheme structure to the space constructed by Alexeev. His construction was restricted to components of the moduli space where the general element is irreducible, as one cannot apply the minimal model progrem na\"ively to slc pairs in general \cite{JK-example}, which leads us to use Koll\'ar's gluing theory. Moreover by applying Koll\'ar's theory of hulls \cite{kollar-hulls-husks}, we are able to give the modul space a unique scheme structure. 

We emphasise that we are not really proving any new technical results in the minimal model program. Instead, our work can be seen as a new application of the existing techniques.  In the absolute case  $Y=\{pt\}$, our construction reduces to the construction of the moduli space of stable slc models, originating in the work of Koll\'ar-Shepherd Barron \cite{KSB} and Alexeev \cite{Alexeev-moduli-KSBA,va-stablemaps}, and we refer to \cite{HK-book,jk-moduli-survey} for a survey of this construction in this case. 

We begin with the definition of the moduli functor of stable maps, which reduces to \cite[Definition 29]{jk-moduli-survey} in the  absolute case $Y=\{pt\}$.

Let $(\X,\D)\to S\times Y$ be a family, flat over $S$, whose fibres over $S$ are stable maps. For a coherent sheaf $F$ on $\X$, denote by $F^{[m]}$ the reflexive hull of $F^{\otimes m}$. Since there is a subscheme $Z\subset \X$ satisfying $Z\cap \X_s$ has codimension two for all $s\in S$ and $(\omega_{\X/S}\otimes \D \otimes p^*H)^{\otimes m}$ is locally free on $X\backslash Z$ for each $m$, this sheaf admits a reflexive hull. It is not true in general that one has an isomorphism \begin{equation}\label{kollar-example}(\omega_{\X/S}\otimes \scO_{\X}(\D) \otimes p^*H)^{[m]}|_{\X_s} \cong (\omega_{\X_s}\otimes \scO_{\X_s}(\D_s) \otimes (p^*H)_s)^{[m]},\end{equation} and the main subtlety in the definition of the moduli functor is to impose a condition on the admissible families such that this property holds.

\begin{definition} Fix an integer valued function $h(m)$. We define the \emph{moduli functor of stable maps} to be: \[
\scM(S)=\left\{
\begin{aligned}
&\text{Projective morphisms $(\X,\D)\to S\times Y$ such that:}\\ 
&\text{$(i)$ $\X\to S$ and $\D\to S$ are flat,}\\
&\text{$(ii)$ the fibres over each $s \in S$ are stable maps,}\\
&\text{$(iii)$ the Hilbert function of each fibre}\\
&\text{$\chi(\X_s, (\omega_{\X_s}\otimes \scO_{\X}(\D_s) \otimes H)^{[m]})=h(m)$ is fixed,}\\
&\text{$(iv)$ $(\omega_{\X/S}\otimes \scO_{\X}(\D) \otimes p^*H)^{[m]}$ is flat over $S$ for all $m \in \Z_{>0}$,}\\
&\text{$(v)$ $(\omega_{\X/S}\otimes \scO_{\X}(\D) \otimes p^*H)^{[m]}$ commutes with arbitrary base change,}\\
&\text{modulo isomorphisms over $S$.}
\end{aligned}
\right\}
\]
with morphisms given by taking the pullback.
\end{definition}

\begin{remark} Here two families $(\X,\D)\to S\times Y$ and  $(\X',\D')\to S\times Y$ are isomorphic over $S$ if there exists an isomorphism $(\X,\D)\to Y \cong (\X',\D')\to Y$ over $S$.\end{remark}

Condition $(v)$ is an adaptation of Koll\'ar's condition to our setting. It means that, for a family  $\X\to S$ as above and an arbitrary morphism $\alpha: T \to S$, for all $m\in\Z$ we have \[ \alpha_X^*(\omega_{\X/S}\otimes \scO_{\X}(\D) \otimes p^*H)^{[m]}_{\X/S} \cong (\omega_{\X/T}\otimes \scO_{\X_T}(\D) \otimes p^*H)^{[m]}.\] Here $\X_T= \X \times_T S$ is the fibre product, while by $\scO_{\X_T}(\D)$ we mean the pullback of the sheaf $\scO_{\X}(\D)$ to $\X_T$. Remark that when $T={s} \in S$ with $\alpha_T$ the inclusion, Koll\'ar's condition ensures the isomorphism (\ref{kollar-example}) exists. 

\begin{remark}\label{kollar-condition-remark} The last two conditions in the definition of the moduli functor are automatic over a reduced base, see \cite[Definition 28]{jk-moduli-survey} and the preceding discussion. Note that the extra term $p^*H$ in the various sheaves in the last two conditions in the definition of the moduli functor plays no role, since it is a locally free on $Y$. However it seems more convenient to include it in order to apply results in the literature directly.\end{remark}

To prove Theorem \ref{moduli-space-existence}, we follow the usual strategy of proving various properties of the moduli functor. Namely, we will prove separatedness, properness, local closedness, boundedness and the finiteness of the automorphism group. Applying the general theory of \cite{KM,kollar-quotients} will then result in the representability of the moduli functor as a separated alegraic space of finite type; we then appeal to a result of Alexeev to obtain projectivity \cite[Theorem 4.2]{va-stablemaps}.

We begin by proving separatedness of the moduli functor, which means that for each family over a  punctured curve $C^0$, a possible extension to $C$ is unique.

\begin{proposition}\label{prop-separated} The moduli functor of stable maps is separated. \end{proposition}

\begin{proof} The proof is identical to the case $Y=\{pt\}$, however as we are not aware reference for this in the general slc setting we include the proof.

Let $(C,0)$ be a pointed curve, and let $(\X,\D)\to C\times Y$ and $(\X',\D') \to C\times Y$ be two families of stable maps which are isomorphic away from $0 \in C$. We wish to show $(\X_0,\D_0)\to Y$ is isomorphic to  $(\X'_0,\D_0')\to Y$, where by stability $K_{\X_0}+\D_0$ and $K_{\X'_0}+\D'_0$ are relatively ample. It is enough to show that the isomorphism between $(\X,\D)$ and $(\X',\D')$ over the punctured curve $C^0$ extends to an isomorphism over all of $C$, as once this has been prove the maps to $Y$ must be equal since they agree over the preimage of $C\setminus \{0\}$ which is open.

As each fibre of the families is slc, inversion of adjunction applies to give that the pairs $(\X,\X_0+\D)$ and $(\X',\X'_0+\D')$ are themselves slc \cite[Lemma 2.12]{Patakfalvi-fibered} \cite{Kaw}. Thus it suffices to show that the slc pairs $(\X,\X_0+\D)$ and $(\X',\X'_0+\D')$ are isomorphic. 

We now reduce to the normal case for pairs. Take the normalisations $\nu: \bar{\X}\to \X$ and $\nu': \bar{\X}'\to \X'$ with conductors $\bar{F}, \bar{F}'$. The pairs $(\bar{\X},\bar{\X_0}+\bar{\D}+\bar{F})$ and $(\bar{\X}',\bar{\X_0}'+\bar{\D}'+\bar{F}')$ are lc pairs which are both canonical models of a common resolution. Thus they are isomorphic by uniqueness of canonical models for lc pairs \cite[Theorem 3.52]{KM-book}. But $(\X,\X_0+\D)$ and $(\X',\X'_0+\D')$ are determined by their normalisations together with the involutions  $\tau: (\bar F^{\nu}, \Diff_{\bar F^{\nu}}(\bar D)) \to (\bar F^{\nu}, \Diff_{\bar F^{\nu}}(\bar D))$ and $\tau': (\bar F'^{\nu}, \Diff_{\bar F'^{\nu}}(\bar D'))\to (\bar F'^{\nu}, \Diff_{\bar F'^{\nu}}(\bar D'))$. Remark that the involutions agree, as they are morphisms which agree away from $0\in C$, which has codimension one preimage in the conductors. Thus $(\X,\X_0+\D)$ and $(\X',\X'_0+\D')$ are isomorphic, as required.
\end{proof}

\begin{remark}The reason one cannot prove separatedness directly using the argument in the irreducible case is that the canonical ring of an slc variety is not finitely generated in general \cite{JK-example}.
\end{remark}

To prove properness, we will first need the following, which is proved in an essentially identical way to the surface case \cite[Lemma 2.23]{va-stablemaps}.

\begin{lemma}\label{relative-to-absolute-ampleness} Let $p: (X,D)\to (Y,M)$ be such that $X$ is slc and $K_X+D$ is $p$-ample. Then $K_X+D+p^*H$ is ample, where $H$ is ``sufficiently ample'' in the sense of Remark \ref{rmk:suffample}. \end{lemma}

\begin{proof} We show $K_X+D+2n p^*M$ is nef, which implies the statement by relative ampleness of $K_X+D$ and the fact that $H-2nM$ is ample by the definition of sufficient ampleness given in Remark \ref{rmk:suffample}. 

Suppose not, so that there is a curve $C$ such that $(K_X+D+2n p^*M).C<0$. As $K_X+D$ is relatively ample and $p^*M$ is semi-ample, $C$ cannot map to a point. By Fujino's version \cite{fujino} of Mori's theorem on the length of extremal rays \cite{mori-extremal} for slc varieties, we know that $(K_X+D).C \geq -2n$. As $C$ does not map to a point, we have $p^*H.C=H.p_*C \geq 1$, hence $(K_X+D+2n p^*M).C\geq0$, as required. \end{proof}

We now proceed to the proof of properness. Recall this entails proving that for an arbitrary family of stable maps over a smooth punctured curve $C^0\subset C$, there exists an extension to some $C'$, where $C'\to C$ is a finite map branched over $0$.

\begin{proposition}\label{prop-proper} The moduli functor of stable maps is proper. \end{proposition}

\begin{proof}

This is a variant of \cite[Theorem 1.5]{HX-lc-closures} which is proven in the absolute case $Y=\{pt\}$. We split the proof into two cases: the first when the general fibre is log canonical (in particular, irreducible), the second is the slc case. Again, the reason is that the log canonical ring of an slc pair may not be finitely generated \cite{JK-example}, so a different argument is needed using Koll\'ar's gluing theory.

$(i)$ (the general fibre is log canonical) 

We apply the valuative criterion for properness, so let $(C,0)$ be a smooth pointed curve and $(\X^0,\D^0)\to C^0 \times Y$ be a family of stable maps. We can complete this to some family $(\X,\D)\to C$, which may not admit a map to $Y$ extending the given one away from $0$. By semistable reduction \cite[Theorem 7.17]{KM-book}, there exists a finite map $C'\to C$, branched over $0 \in C$, and a resolution $\tilde{\X}\to \X'_{\nu} \to C'$ with $\X'_{\nu}$ the normalisation of $\X' = \X\times_C C'$ such that $\tilde{\X}_0+\tilde{D}$ is a reduced relatively snc divisor, where $\tilde \D$ is the pullback. Passing to a resolution of indeterminacy of the induced rational map $\tilde{\X}\dashrightarrow C' \times Y$ if necessary, we can assume that $\tilde{\X}$ itself admits a morphism to $C'\times Y$. The important point is that the relative log canonical model of $(\tilde{\X}^0,\tilde{D}^0)\to C'^{0} \times Y$ is $(\X^0,\D^0)\to C'^0 \times Y$.

Consider the log smooth, hence dlt, pair $(\tilde{X}, \tilde{X}_0+\tilde{D})$. By the fundamental result of Hacon-Xu \cite[Theorem 1.1]{HX-lc-closures} on the existence of log canonical closures, since the relative log canonical model of $(\tilde{\X}^0,\tilde{D}^0)$ over $C'^{0} \times Y$ exists, and  $(\tilde{X}, \tilde{X}_0+\tilde{D})$ is a dlt pair, the relative log canonical model of $(\tilde{X}, \tilde{X}_0+\tilde{D})$ also exists. Write this model as $(\bar{\X},\bar{\X}_0+\bar{\D}) \to C'\times Y$. Since $(\bar{\X},\bar{\X}_0+\bar{\D})$ is log canonical, adjunction implies $(\bar{\X}_0,\bar{\D}_0)$ is semi-log canonical. As $K_{\X_0}+\bar{\D}_0$ is ample over $Y$, Lemma \ref{relative-to-absolute-ampleness} and Remark \ref{rmk:suffample} imply that $(\bar{\X}_0,\bar{\D}_0)\to Y$ is a stable map. By Remark \ref{kollar-condition-remark}, since we are working over a reduced base, the final two conditions in definition of the moduli functor are satsified. Thus $(\bar{\X},\bar{\D})\to C' \times Y$ is the family we seek.

$(ii)$ (the general case) 

Assume we have a family $(\X^0,\D^0)\to C^0 \times Y$ where all fibres over $C^0$ are stable maps. Let $\nu: (\bar{\X}^{0},\bar{\D}^{0}) = \cup_j (\bar{\X}_j^{0},\bar{\D}^{0}_j)\to (\X^0,\D^{0})$ be the normalisation, where $\bar{\X}_j^{0}$ are the connected components, and let $\bar{F}^0= \cup_j \bar{F}^0_j$ be the conductor. Let $\tau^0$ be the involution of the normalisation of $\bar{F}^0$, so that the data $(\bar{\X}^{0},\bar{F}^{0}+\bar{D}^0)$ together with the involution $\tau^0: (\bar{F}^{0,\nu},\Diff_{\bar{F}^{0,\nu}}(\bar{D}^0))\to (\bar{F}^{0,\nu},\Diff_{\bar{F}^{0,\nu}}(\bar{D}^0))$ determines $(\X^0,\D^0)\to C^0 \times Y$. 

By $(i)$, after passing to a base change of $C$, each $(\bar{\X}_j^{0},\bar{F}_j^{0}+\bar{D}_j^0)$ admits a log canonical closure $(\bar{\X}_j,\bar{F}_j+\bar{D}_j)\to C'\times Y$ extending $(\bar{\X}_j^{0},\bar{F}_j^{0}+\bar{D}_j^0)\to C'^0\times Y$. In particular, $(\bar\X_{j,0},\bar{F}_{j,0}+\bar{D}_{j,0})$ is log canonical for each $j$.

Set $\bar{\X} =  \cup_j \bar{\X}_j^{0}$, and let $\bar{F}$ be the closure of $\bar{F}^{0}$ in $\bar{X}$. We need to extend the involution to the normalisation of $\bar{F}$. Since $(\bar{F}^{n},\Diff_{\bar{F}^n}(\bar{D}^n))\to C\times Y$ is a family of stable varieties, the fibre over $0\in C$ must be unique. Thus the existence of the involution over $C^0$ implies the existence of the involution over $C$, as in \cite[Theorem 1.5, Step 2]{HX-lc-closures}. 

Next, following \cite[Section 7, Step 3]{HX-lc-closures}, we apply Koll\'ar's gluing theory \cite[Section 5]{kollar-singularities}. A direct application of  \cite[Section 7, Step 3]{HX-lc-closures} gives that $(\bar{X},\bar{F},\bar{D},\tau)$ is the gluing data of some $(\hat\X,\hat{\D})\to C'$, and what remains to be proved is that  $(\hat\X,\hat{\D})$ actually admits a map to $C'\times Y$. But the map to $Y$ exists for precisely the same reason that the map to $C'$ exists, which is the universal property of $(\hat\X,\hat{\D})$ as defined by Koll\'ar \cite[Definition 9.4]{kollar-singularities}. Indeed, the induced map $(\bar{F}^{\nu},\Diff_{\bar{F}^{\nu}}(\bar{D}))\to Y$ is $\tau$-invariant, hence $(\hat\X,\hat{\D})$ admits a map to $C'\times Y$ as required.

Summing up, we obtain a stable map $(\hat\X_0,\hat{\D}_0)\to Y$ as required.\end{proof}

Next we show that each stable map has finite automorphism group.
\begin{definition} We define the \emph{automorphism group} $\Aut(p)$ of a map $p: (X,L)\to (Y,H)$ to be the automorphisms of $(X,L)$, which cover $p$. \end{definition}
 
From another point of view, these are isomorphisms of the graph $\Gamma_p\subset X\times Y$ of $p$ that lift to the polarisation  $L|_{\Gamma_p} = (K_X + D + p^*H)|_{\Gamma_p}$.  The following is due to Alexeev \cite[Theorem 3.23 (3)]{va-stablemaps}.

\begin{proposition}\label{prop-aut} The moduli functor of stable maps has finite automorphism group. \end{proposition}

\begin{proof} 

Alexeev proves this by noting that $\Aut(p)$ equals $\Aut(X,p^*G)$ for a general $G\sim_{\Q}H$. As $p^*H$ is semi-ample, the pair $(X,G+D)$ is slc for general $G$ and one concludes by ampleness of $K_X +D+p^* H$ (by Remark \ref{rmk:suffample}).

One can alternatively adapt the direct proof of Hacon-Xu \cite[Lemma 3.4]{HX-automorphisms} in the absolute case $Y = \{pt\}$ with $X$ normal. First note that one can assume normality by taking the normalisation and working with automorphisms preserving the conductor. Next, it suffices to show $\Aut(p)$ contains no copies of $\C^*$ or $\C_{+}$. Taking a copy of either, the closure of orbit of a point $x \in X$ gives a curve $C \subset X$. Under these hypotheses, Hacon-Xu show that for a general $x\in X$, we have $(K_X+D).C \leq 0$. The automorphisms of the map $p$ correspond to curves which map to a point in $Y$. As $K_X+D$ is relatively ample, this gives a contradiction and so $\Aut(p)$ is finite. \end{proof}

We next prove boundedness of the moduli functor. Each stable map defines a point in a product of Hilbert schemes of subschemes of $\pr^N\times Y$ using the natural polarisation and taking the divisor into account. To prove boundedness, we need to prove that $N$ can be chosen independent of the map in question. This follows directly from Alexeev's proof in the case of surfaces, which we briefly recall.

\begin{proposition}\cite{va-stablemaps}\label{prop-bounded} The set of stable maps is bounded. \end{proposition}

\begin{proof}

As explained by Alexeev, this follows from the absolute case $Y=\{pt\}$ for pairs, which is due to Hacon-McKernan-Xu \cite[Theorem 1.1]{HMX-boundedness}. We repeat Alexeev's argument for the reader's convenience.

Since $p^*H$ is semiample, the pair $(X,G+D)$ is slc for a general $G\sim_{\Q} p^*H$. By \cite[Theorem 1.1]{HMX-boundedness}, it follows that there is an $m \gg 0$ such that each $(X,G+D)$ with fixed Euler characteristic $\chi(X,m(K_X+D+p^*H))$ satisfies $m(K_X+G+p^*H)$ is very ample and is without higher cohomology. Here we are using that as line bundles, $K_X+D+p^*H \cong K_X+D+p^*G$. A map $p: X \to Y$ is determined by its graph $$\Gamma_p \subset \pr(H^0(X,m(K_X+D+p^*H))) \times Y.$$ Thus each graph is parameterised by a point in a Hilbert scheme of subschemes of this product. As the Hilbert polynomial of the graph is equal to the fixed quantity $\chi(X,m(K_X+D+p^*H))$, this embeds each graph in a single a Hilbert scheme. The same argument applies to the divisors, giving the required boundedness. \end{proof}

The next step in producing the moduli space is to prove the moduli functor is locally closed. From the previous Proposition, each stable map is parameterised by a point in some fixed Hilbert scheme of subschemes of $\pr^N\times Y$. To prove local closedness of the moduli functor, we must show that points in this Hilbert scheme parameterising stable maps form a locally closed subscheme.

\begin{proposition}\label{prop-locallyclosed} The moduli functor of stable maps is locally closed. \end{proposition}

\begin{proof} 

Fix a subscheme $Z\subset \Hilb(\pr^N\times Y)\times \Hilb(\pr^N\times Y)$, where the presence of two Hilbert schemes is due to the presence of divisors. We wish to show that the locus inside $Z$ parameterising stable maps is locally closed.

Firstly we may replace $Z$ with the locus in $Z$ parameterising graphs of morphisms to $Y$, since being a graph of a morphism is an open condition \cite[p96]{kollar-rational-curves}. By the flattening decomposition theorem \cite[Lecture 8]{DM}, we then obtain a locally closed decomposition of $Z$ such that each component parameterises flat families. Next, being reduced and $S^2$ are open conditions, and similarly being Gorenstein in codimension one is an open condition. From here, we obtain that being slc is also an open condition by a result of Alexeev \cite[Appendix A]{AH}. 

We can therefore consider an arbitrary projective family $(\X,\D)\to S$ of graphs of maps to $(Y,H)$ with each fibre satisfying the above properties. By Koll\'ar's theory of hulls and husks, the condition that the family satisfies Koll\'ar's condition is then locally closed \cite[Corollary 25]{kollar-hulls-husks} (see also \cite[Theorem 31]{jk-moduli-survey}). Indeed, one can apply \cite[Corollary 25]{kollar-hulls-husks} directly by Remark \ref{kollar-condition-remark}. 

The final point to prove is that the polarisation agreeing with $K_X+D+p^*H$ is a locally closed condition, which is due to Viehweg \cite[Lemma 1.19]{EV}.

\end{proof}

We now produce the moduli space as an algebraic space. So far we have produced a locally closed subscheme for which each point parameterises a stable map. These points are only unique up to the obvious projective transformations, thus we take a quotient to remove this ambiguity and form a genuine moduli space.

\begin{proposition} The moduli functor of stable maps is coarsely represented by a proper, separated algebraic space of finite type. \end{proposition}

\begin{proof}

From the properties already established, this follows immediately from the main results of \cite{KM,kollar-quotients}. 

By the main results of \cite{KM,kollar-quotients}, the quotient of a separated algebraic space $\scH$ by a reductive group, such that each point has finite stabiliser, exists as a separated algebraic space. We take $\scH$ to be the locally closed subscheme of the Hilbert scheme produced in Propostion \ref{prop-locallyclosed}, and the reductive group to be the automorphisms of the corresponding projective space. Separatedness of this scheme follows from Proposition \ref{prop-separated}, while Proposition \ref{prop-aut} implies each point has finite stabiliser. That the resulting algebraic space is of finite type  and proper follows from the boundedness and properness proved in Proposition \ref{prop-bounded} and Proposition \ref{prop-proper} respectively.

 \end{proof}

The last point is to show the moduli space produced above is projective. This follows from Alexeev's work, which uses the technique of Koll\'ar \cite{JK-projectivity}.

\begin{theorem}\cite[Theorem 4.2]{va-stablemaps} The moduli space of stable maps is projective. \end{theorem}

Alexeev's proof appeals to a semipositivity theorem of Koll\'ar which applies for surfaces \cite[Proposition 4.7]{JK-projectivity}, the analogous result for higher dimensional pairs is due to Fujino \cite[Theorem 1.12]{fujino-moduli}.

A further consequence of this is that the natural analogue of the CM line bundle, as defined in equation \eqref{eq:defCMlinebundle} in the absolute case, on the moduli space of stable maps defined in Section \ref{sec:fibrations} is nef, as it is the leading order term in a Knudson-Mumford expansion of line bundles which are proved to be ample by Alexeev. We expect that the CM line bundle is actually ample; in the absolute case, this has been proved by Patakfalvi-Xu \cite{PX}.

\subsection{Enumerative geometry}\label{sec:GW}

Given the existence of the moduli space of canonically polarised stable maps $\M(Y,H)$, it is natural to ask whether there is an analogue of Gromov-Witten invariants in higher dimensions. We  first briefly recall the definition of Gromov-Witten invariants, for which we refer to \cite{FP} for further details. 

Let $\overline M_{g,m}(Y,H)$ be the (proper) Kontsevich moduli space of stable maps with domain of genus $g$ and $m$ marked points. Then for $1\leq i \leq m$ one has maps $\ev_i: \overline M_{g,m}(Y,H)\to Y$ defined by $\ev_i(X,p_1,\hdots p_m) = p_i$. Through these maps, one can pullback cycles from $Y$ to $\overline M_{g,n}(Y,H)$. The moduli space $\overline M_{g,m}(Y,H)$ admits a virtual fundamental class $[\overline M_{g,m}(Y,H)]^{vir}$, and one defines the Gromov-Witten invariants by integrating $$\GW(\alpha_1,\hdots,\alpha_k) = \int_{[\overline M_{g,m}(Y,H)]^{vir}} \alpha_1\cdot \hdots \cdot \alpha_k,$$ where the $\alpha_k$ are cycles on $Y$ pulled back via the evaluation maps. The use of the virtual fundamental class ensures deformation invariance. Through the natural map $\pi: \overline M_{g,m}(Y,H) \to \overline M_{g,m}$ obtained by stabilising, the Gromov-Witten invariant are computed on the orbifold $\overline M_{g,m}$ (where one can intersect cycles).

In \cite[Question 7.1]{va-icm}, Alexeev suggests an definition of Gromov-Witten invariants in higher dimensions, by intersecting the $D_i$ to obtain a zero cycle on $Y$ and mimicking the above definition when the domain is a curve. Here we suggest an alternative approach.

Consider the moduli space of stable maps $p: (X,D) \to (Y,H)$ where $X$ is $n$-dimensional and $D = \sum_{i=1}^m D_i$ is a Weil divisor. Then the analogue of the evaluation map sends $p: (X,D) \to (Y,H)$ to $D_i$, which defines a point in a Hilbert scheme  $\Hilb_i$ of subschemes of $Y$. Denote this map by $ev_i: \M(Y,H)\to \Hilb_i$. The moduli space $\M$ of canonically polarised varieties is, in general, highly singular, so one cannot intersect cycles on it. However, one can intersect line bundles on an arbitrary scheme. Thus it is natural to take line bundles $L_1,\hdots L_k$ on $\M(Y,H)$ and $H_1,\hdots,H_{m}$ on $\Hilb_i$, where $\dim \M(Y,H) = k+m$, and define $$\GW(L_1,\hdots L_k,H_1,\hdots H_m) = \int_{\M(Y,H)} L_1\cdot\hdots \cdot L_k \cdot \ev_1^*H_1\cdot \hdots \cdot \ev_m^*H_m.$$ One can similarly pullback multiple line bundles using the same evaluation map. A natural choice of line bundle on $\M(Y,H)$ is the CM line bundle. Note that a $\Q$-Cartier divisor on $\Hilb_i$ defines a line bundle, so one obtains line bundles on $\Hilb_i$ by picking certain families of subschemes of $Y$. Another interesting way to produce line bundles on $\Hilb_i$ is to take any line bundle $T$ on $Y$, and take $H_i$ to be the induced CM line bundle $T_{CM}$ on $\Hilb_i$ defined as in equation \eqref{eq:defCMlinebundle}.

While this does define a numerical invariant, their geometric interpretation is not transparent. The line bundles $H_i$ define divisors on $\Hilb_i$, and thus define families of subschemes of $Y$. The higher dimensional Gromov-Witten invariants may be related to the count of varieties $X$ intersecting these families of subschemes. 

Ideally, one would replace the above with an integral over a virtual fundamental class, in the hope of making the above invariants deformation invariant. Unfortunately, this seems out of reach with present techniques, essentially because one no longer obtains a two term obstruction complex when $\dim X>1$.  

The moduli spaces $\M(Y,H)$ are constructed rather non-explicitly, and even when $Y$ is a point there are very few explicit examples of the moduli space. Thus it seems somewhat hopeless to compute the above invariants in any cases at present. A variant of the above construction would be to construct a moduli space of K-stable \emph{Fano} maps, i.e. with $-K_X-D-p^*H$ ample. In the absolute case, the moduli space of K-stable Fanos can often be constructed quite explicitly \cite{OSS}, thus it seems much more reasonable that one could compute the analogous invariants in the Fano case.

\bibliography{stablemaps}
\bibliographystyle{amsplain}

\vspace{4mm}

\end{document}